\documentclass[11pt]{amsart}
\usepackage{amsmath}
\usepackage{amssymb,amscd}
\usepackage{amsfonts}
\usepackage{euscript}

\usepackage{graphicx}
\usepackage{color}

\DeclareMathAlphabet{\mathpzc}{OT1}{pzc}{m}{it}

\numberwithin{equation}{section}

\newtheorem{theorem}{Theorem}[section]

\newtheorem{corollary}[theorem]{Corollary}

\newtheorem{lemma}[theorem]{Lemma}
\newtheorem{prop}[theorem]{Proposition}

\theoremstyle{definition}
\newtheorem{defn}[theorem]{Definition}

\newtheorem{example}[theorem]{Example}
\newtheorem{ques}[theorem]{Question}
\newtheorem{remark}[theorem]{Remark}

\def \begineq{\begin{equation}}
\def \endeq{\end{equation}}

\def \bb{\mathbb}

\def \CC{{\bb{C}}}

\def \PP{{\bb{P}}}

\def \RR{{\bb{R}}}

\def \ZZ{{\bb{Z}}}

\def \({\left(}
\def \){\right)}
\def \<{\langle}
\def \>{\rangle}
\def \bar{\overline}

\def \tensor{\otimes}

\def \Aut{{\rm Aut}}
\def \Diff{{\rm Diff}}

\def \img{{\rm Im}}

\def\VC{{\rm VC}}
\def \Id{\rm Id}
\def \Sing{\rm Sing}
\def \Fr{\rm Fr}

\def \qed{\hfill $\square$ \vspace{0.03in}}

\begin{document}

\title{Equivariant cobordism of torus orbifolds}

\author[S. Sarkar]{Soumen Sarkar}

\address{Department of Mathematics, Indian Institute of Technology Madras, Chennai, 600036, India.}

\email{soumen@iitm.ac.in}

\author[D. Y. Suh]{Dong Youp Suh}

\address{Department of Mathematical Sciences, KAIST, Daejeon 305701, Republic of Korea}

\email{dysuh@kaist.ac.kr}


\thanks{The authors would like to thank Mainak Poddar and Nigel Ray for many helpful discussion.
The first author thanks Indian Institute of Sciences, Pacific Institute for Mathematical Sciences and
University of Regina for support. He also thanks Indian Statistical Institute, 
Kolkata and Institute of Mathematical Sciences for supporting his visiting fellowship.
The second author is partially supported by the National Research Foundation of Korea(NRF) grant funded by the Korea government(MSIP) (No. 2016R1A2B4010823).
}

\abstract
Torus orbifolds are topological generalization of symplectic toric orbifolds.
We give a construction of smooth orbifolds with torus actions whose boundary
is a disjoint union of torus orbifolds using toric topological method. As a
result, we show that any orientable locally standard torus orbifold is
equivariantly cobordant to some copies of orbifold complex projective spaces. 
We also discuss some further equivariant cobordism results including the cases
when torus orbifolds are actually torus manifolds.
\endabstract

\maketitle

\section{Introduction}\label{intro}
Cobordism, a fundamental concept in topology,  was first introduced by Lev Pontryagin
in his pioneering work on a classifiction of manifolds, \cite{Pon}. There are two
known definitions of bordism; one is geometric and another is homotopy theoretic.
Thom constructions were used in \cite{Tho} to show that cobordism
groups could be computed through homotopy theory in the early  1950's. Now the
oriented, non-oriented and complex cobordism rings of manifolds are completely
known in respective category. Conner and Floyed generalized this definition
to the equivariant category to study transformation groups in the beginning of sixties.
After few years, tom Dieck introduced homotopy theoretic cobordism in equivariant category. 
But, in the equivariant category these two cobordism theories are not equivalent, 
see \cite{Sinha}.
Even though there have been many developments such as \cite{Katz, LT, Madsen, Stong,
Stong2, tom}, the equivariant cobordism rings are not determined for any nontrivial groups. 
The main reason is that the Thom transversality theorem may not hold in equivariant 
category, and hence the equivariant cobordism cannot be reduced to  (equivariant)
 homotopy theory. Here, we consider the geometric definition of equivariant cobordism. 

Orbifold, which is a natural generalization of manifold, was introduced
by I. Satake in \cite{Sat} where it was called $V$-manifold. The study
of oriented cobordism of orbifold first appeared in \cite{Dru1} where the
author introduced a complete set of invariants which determine the oriented
cobordism classes up to torsion.
Torus orbifold is a generalization of symplectic toric orbifold \cite{LeTo} and it was introduced in \cite{HM} where some topological invariants of these spaces were studied.
Some examples of torus orbifolds are the class of quasitoric orbifolds whose topological invariants are studied in \cite{PS}. Briefly, a torus orbifold is a closed orientable $2n$-dimensional orbifold having an effective action of a real $n$-dimensional compact torus with nonempty fixed point set. In addition if the torus action is locally standard and the orbit space is a simple polytope then it is called a quasitoric orbifold. 
When the orbifold singularities of torus orbifolds are trivial, they are called torus manifolds, see \cite{HM} for topological properties and connections between torus manifolds and multifans. 
%An orbifold with torus orbifold boundary is an orbifold with boundary where the boundary is a disjoint union of some torus orbifolds.

In this paper we  modify the basic construction of 1.5 in \cite{DJ}
to produce effective orbifolds with torus actions whose boundary consist of disjoint copies of locally standard torus orbifolds. 
Using this construction, we can show in Theorem \ref{coborthm} that any locally standard torus orbifold  is equivariantly cobordant
to a disjoint union of some  orbifold complex projective spaces.
% We remark that this cobordism result is in homeomorphic category unlike the cases in \cite{Sar2, Sar3}. Nevertheless our proofs are constructive so we may not need to be worry about any differentiability issue. 
We also have some more equivariant cobordism results on locally standard torus
orbifolds and manifolds, see Theorem \ref{thm_4d-loc-tor} and \ref{bdd_orb_hir_sur}.  

The article is organized as follows.
Following \cite{ALR}, we recall the definitions and some facts concerning effective orbifolds in Section \ref{deforb}.
In Section \ref{tororb},  we define locally standard torus orbifold $X$ over a nice manifold with corners $Q$,
and from $X$ we extract  two data one of which is a rational characteristic function $\lambda$, and the other is a principal $T^n$-bundle $\tau$ over $Q$, 
which we call  combinatorial and topological data of $X$. 
We also show that from a combinatorial and topological data $\lambda$ and $\tau$  over $Q$, we can actually construct a locally standard torus orbifold $X(Q,\lambda,\tau)$.
We then show that $X$ and $X(Q,\lambda,\tau)$ are equivariantly homeomorphic in Theorem~\ref{clasi}. 
All these arguments are appropriate modifications of the basic construction  in \cite{DJ}.
We discuss the generalization of weighted projective spaces in Subsection \ref{gwps}. 
These spaces are called  orbifold complex projective spaces. 
We also define  orbifold Hirzebruch surface in Subsection~\ref{orb_hir_sur}, which is an orbifold modification of Hirzebruch surface.
In Section~\ref{def}, we give explicit  construction of $(2n+1)$-dimensional smooth orbifolds with $T^n$-action whose boundaries are disjoint union of 
locally standard torus orbifolds.
 In Section~\ref{cobor}, we show that any torus orbifold with a locally standard torus
action is equivariantly cobordant to some copies of orbifold complex projective spaces in Theorem~\ref{coborthm}. 
We also give several further equivariant cobordism results of torus manifolds and orbifolds including orbifold Hirzebruch surfaces.

%\tableofcontents

\section{Classical effective orbifolds}\label{deforb}
%Orbifolds are introduced by Satake \cite{Sat}, who called them $V$-manifolds. 
An orbifold is a singular space that locally looks like the quotient of an open subset of a Euclidean space by an action of a finite group. We recall the definitions and some facts concerning effective orbifolds from \cite{ALR}. The reader may also consult \cite{MM} for an excellent exposition of the foundations of the theory of differentiable orbifolds. 
%For notational convenience we may denote an orbifold and its underlying topological space by the same letter in normal font. 

\begin{defn}
 Let $X$ be a Hausdorff topological space.
An $n$-dimensional \emph{effective orbifold~ chart} on an open subset $U \subseteq X $ is given by a triple $(\widetilde{U} , G, \varphi)$ where
\begin{enumerate}
 \item $ \widetilde{U} $ is a connected open subset of $ \RR^n $,
\item $ G$ is a finite subgroup of the self diffeomorphisms of $ \widetilde{U} $, and has an effective $G$-action on $ \widetilde{U} $,
\item $ \varphi $ is a map from $ \widetilde{U} $ to $X$ such that $\varphi$ is a $G$-invariant map inducing 
a homeomorphism from $\widetilde{U}/ G$ onto $U$.
\end{enumerate}
\end{defn}
\begin{defn}
 An \emph{embedding} $\xi \colon (\widetilde{V} , H, \zeta) \to (\widetilde{U} , G, \varphi)$ between two
orbifold charts is a smooth embedding $\xi \colon \widetilde{V} \to \widetilde{U}$ of manifolds such that $ \varphi \circ \xi =\zeta$.
\end{defn}
\begin{defn}\label{def_compatibile_orbifold_chart}
 Two orbifold charts $(\widetilde{V}, H, \zeta)$ on $V = \zeta (\widetilde{V}) \subseteq X$ and
$( \widetilde{U}, G, \varphi)$ on $U = \varphi(\widetilde{U} ) \subseteq X$ with 
 $ V \cap U \neq \emptyset$
are \emph{locally~ compatible} if for any $x \in V \cap U$ there exists an open
neighborhood $ W \subseteq V \cap U $ of $ x$ and
an orbifold chart $(\widetilde{W}, K, \mu)$ on $W$ such that there exist smooth  embeddings
$( \widetilde{W}, K, \mu) \to (\widetilde{V}, H, \zeta) $ and $ (\widetilde{W}, K, \mu) \to (\widetilde{U}, G, \varphi)$.
\end{defn}
\begin{defn}
 An \emph{effective orbifold  atlas} on $X$ is a family $\mathcal{U} = \{(\widetilde{U}, G, \varphi)\}$ of locally compatible effective orbifold charts
such that $\{\varphi(\widetilde{U})\mid \widetilde{U} \in \mathcal{U}\}$ is an open cover of $ X $.
\end{defn}

An atlas $\mathcal{V}$ is a {\em refinement} of an atlas $\mathcal{U}$ if for
any chart $(\widetilde{V}, H, \zeta) \in \mathcal{V}$ there exists an embedding
$\xi \colon (\widetilde{V}, H, \zeta) \to (\widetilde{U}, G, \varphi)$ into some
chart $ (\widetilde{U}, G, \varphi) \in \mathcal{U}$.

Two orbifold atlases are said to be {\em equivalent } if they have a common refinement.
We denote the equivalence class of an atlas $\mathcal{U}$ by $[\mathcal{U}]$.

\begin{defn}
Let $X$ be a para-compact Hausdorff space equipped with an equivalence class $[\mathcal{U}]$ of
$n$-dimensional effective orbifold atlases. The pair $(X, \mathcal{U})$, denoted by $\mathcal{X}$, is called
an \emph{effective  orbifold} of dimension $n$. The space $|\mathcal{X}|:=X$ is called the \emph{underlying
space} of $\mathcal{X}$.
\end{defn}

For simplicity of notation we may not distinguish $\mathcal{X}$ and $X$ when the orbifold atlas on $X$ is clear. Throughout this paper we assume that all orbifolds are effective.

\begin{defn}
Let $\mathcal{X} = (X, \mathcal{U})$ be an orbifold and $x \in X$. Let $( \widetilde{U}, G, \varphi)$ be
an orbifold chart so that $x = \varphi(\widetilde{x}) \in \varphi(\widetilde{U}) \subset X$. The \emph{local~ group} at $x$
is defined to be the group $G_x = \{g \in G \mid g \cdot \widetilde{x} = \widetilde{x}\}.$
\end{defn}

The group $G_x$ is uniquely determined up to isomorphism. We use the notion of local
group to define the singular set of the orbifold $\mathcal{X}$ as follows.
A point $x \in X $ is called a \emph{nonsingular point} (or {\em smooth point}) if the group $G_x$ is trivial, and  otherwise $x$ is called a \emph{singular point}.
The set of singular points of an orbifold $\mathcal{X} = (X, \mathcal{U})$ is called the (orbifold) \emph{singular set}, denoted
by $\varSigma \mathcal{X}$. That is,
\begin{equation}
\varSigma \mathcal{X} = \{x \in X \mid G_x \neq 1\}.
\end{equation}

\begin{example}
Let $G$ be a finite subgroup of $GL_n (\CC)$ and let $X = \CC^n /G$.
This is an orbifold complex manifold called a quotient singularity. Orbifold $X$ has the structure
of an algebraic variety, arising from the algebra of $G$-invariant polynomials on $\CC^n$.
\qed
%If $H \subset SL_n (\CC)$, the quotient $\CC^n /H$ is called $Gorenstein$.
\end{example}

\begin{example} Consider
$S^{2n+1} = \big\{(z_0 , \ldots , z_n ) \in \CC^{n+1} \mid {\displaystyle \sum_{i=0}^n} ~|z_i|^2 = 1\big \},$
and the action of the circle group $S^1$ is  defined by
\begin{equation}
\alpha (z_0, \ldots, z_n) = (\alpha^{a_0} z_0, \ldots, \alpha^{a_n} z_n )
\end{equation}
for $\alpha \in S^1$, where  integers $a_i$'s are relatively prime. The quotient space
\begin{equation}
 W\PP(a_0, \ldots, a_n) = S^{2n+1}/S^1
\end{equation}
has an orbifold structure, denoted by $\mathcal{W}\PP(a_0, \ldots, a_n)$.
This orbifold is called a {\em weighted projective space} with weights 
$\{a_0, \ldots, a_n\}$. In particular, the orbifold $\mathcal{W}\PP(1, a)$ is
called a {\em teardrop}. \qed

% It is well known that these orbifolds are non-global quotient orbifolds.
\end{example}

Similarly to the definition of a manifold with boundary, we can talk about an orbifold with
boundary. We write $\RR_{\geq 0} = \{x \in \RR ~:~ x \geq 0\}$. 
\begin{defn}
 Orbifold charts of an \emph{orbifold with boundary} $W$ are given by the
compatible triples  $\{(\widetilde{U} , G, \varphi)\} $ where
$ \widetilde{U} \subset \RR^{n-1} \times \RR_{\geq 0}$ for some $n$, $G$ is a finite group acting effectively on $\widetilde{U}$,
and $\{\varphi{(\widetilde{U})}\}$ is an open cover of $W$.
\end{defn}
 The orbifold boundary of $W$, denoted by $\partial{W}$, is the set of points
 $w \in W$ such that $w \in \varphi(\partial{\widetilde{U}})\}$ for some chart
 $(\widetilde{U} , G, \varphi)$ of $W$. We remark that the boundary of an orbifold
 depends on the orbifold chatrs on it. For example, $[0, 1)$ with the trivial chart
 $\{0\}$ is the boundary, but with the chart $((-1, 1), \ZZ_2, \phi)$ where $\ZZ_2$-acts
 by reflection $[0, 1)$ has empty boundary. 
\begin{defn}
 Let  $\mathcal{X} = (X, \mathcal{U})$ and $\mathcal{Y} = (Y, \mathcal{V})$  be two orbifolds.
A map $f \colon X \to Y$ is called an  {\em orbifold map} (respectively {\em orbifold smooth map}) if for any point $x \in X$ there are charts
$(\widetilde{U}, G, \varphi)$ containing $\varphi^{-1}(x)$ and $(\widetilde{V}, H, \zeta)$ containing $\zeta^{-1}(f(x))$, such that $f$ maps $U = \varphi(\widetilde{U})$ into $V = \zeta(\widetilde{V})$ and $f$ can be lifted to a continuous  (resp. smooth) map $\widetilde{f} \colon \widetilde{U} \to \widetilde{V}$ with $\zeta \circ \widetilde{f} = f \circ \varphi$.
\end{defn}
Using this, we can define the notion of  homeomorphism (respectively diffeomorphism) of orbifolds. We note that homeomorphism as topological spaces may not induce a homeomorphism
as orbifolds. For example, a triangle is homeomorphic to a square as a topological space
but not as an orbifold.
\begin{defn} 
Two orbifolds $\mathcal{X} $ and $\mathcal{Y}$ are \emph{homeomorphic} (resp. \emph{diffeomorphic}) if there are orbifold maps
(resp. orbifold smooth  maps) 
 $f \colon X \to Y $ and $g \colon Y \to X $ such that $g \circ f = 1_X$  and $f \circ g = 1_Y$.
\end{defn}

\begin{defn}
An orbifold $\mathcal X=(X, \mathcal U)$ is  {\em orientable} if 
for each chart $(\widetilde{U}, G, \varphi)\in \mathcal U$ the open subset $\widetilde{U}\subset \RR^n$ can be  given 
with an orientation which is invariant 
under the action of $G$,
and each embedding $\xi \colon (\widetilde{W}, K, \mu) \to (\widetilde{U}, G, \varphi)$ of charts in $\mathcal U$ is orientation-preserving.
\end{defn}

Note that the underlying space $X$ of an orbifold $\mathcal X=(X,\mathcal U)$ can be obtained by gluing local charts 
using transition functions in the following way. 
For  two  charts $(\widetilde{U}, G, \varphi)$, $(\widetilde{V}, H, \zeta)\in \mathcal U$ with $x\in U\cap V$  there is a chart $(\widetilde{W}, K, \mu)$  with smooth embeddings
$$\lambda_1 \colon (\widetilde{W}, K, \mu)\to (\widetilde{U}, G, \varphi) \quad \mbox{and} \quad \lambda_2 \colon (\widetilde{W}, K, \mu)\to (\widetilde{V}, H, \zeta)$$ such that $x\in W\subset U\cap V$.
Then $\lambda_2\lambda_1^{-1} \colon \lambda_1(\widetilde{W})\to \lambda_2(\widetilde{W})$ is a $K$-equivariant diffeomorphism.
Thus we can glue $\widetilde{U}/G$ and $\widetilde{V}/H$ by identifying $\varphi(\widetilde{u})\sim\zeta(\widetilde{v})$ if 
$\lambda_2\lambda_1^{-1}(\widetilde{u})=\widetilde{v}$.
Then we have a homeomorphism
$$\Phi \colon \bigsqcup_{\widetilde{U}\in \mathcal U}(\widetilde{U}/G)/\sim\, \longrightarrow X$$
induced from the collection $\{\varphi \colon \widetilde{U} \to X\}$.

We now define the tangent bundle and the frame bundle of an an effective orbifold by gluing local charts using appropriate transition function. We remark that tangent bundles and frame bundles are two of  many  notions originally defined for manifolds which can be extended analogously to orbifolds.
For an orbifold chart $(\widetilde{U}, G,\varphi)$ consider the tangent bundle $T\widetilde{U}$. 
Then it has the induced smooth action of $G$, and the natural projection 
$p \colon T\widetilde{U}/G\to U$ is induced from $\varphi$. 
It is known that for each $x=\varphi(\widetilde{x})\in U$ the fiber
$p^{-1}(x)$ is diffeomorphic to $T_{\widetilde{x}}\widetilde{U}/G_x$.
Thus $p$ is a bundle-like map whose fiber is of the form $\RR^n/G_0$ for some finite subgroup $G_0\subset GL_n(\RR)$.
We thus have $2n$-dimensional orbifold chart $(T\widetilde{U}, G, \pi)$ for each $(\widetilde{U}, G, \phi)\in \mathcal U$ 
where $\pi \colon T\widetilde{U} \to T\widetilde{U}/G$ is the orbit map.
Let 
$$T\mathcal U:=\{(T\widetilde{U}, G, \pi)\mid (\widetilde{U}, G, \phi)\in \mathcal U\}.$$
For  two  charts $(\widetilde{U}, G, \varphi)$, $(\widetilde{V}, H, \zeta)\in \mathcal U$ with $x\in U\cap V$  there is a chart $(\widetilde{W}, K, \mu)$  with smooth embeddings
$\lambda_1 \colon (\widetilde{W}, K, \mu)\to (\widetilde{U}, G, \varphi)$ and $\lambda_2 \colon (\widetilde{W}, K, \mu)\to (\widetilde{V}, H, \zeta)$ such that $x\in W\subset U\cap V$.
We glue  $T\widetilde{U}/G$ and $T\widetilde{V}/H$ using the transition function
\begin{equation}\label{transition funtion}
D_{\widetilde{y}}(\lambda_2\lambda_1^{-1}) \colon T_{\widetilde{y}}\lambda_1(\widetilde W) \to T_{\widetilde{z}}\lambda_2(\widetilde W)
\end{equation}
for $\widetilde{y}\in \lambda_1(\widetilde{W})$ and  $\widetilde{z}=\lambda_2\lambda_1^{-1}(\widetilde y)$
to produce the identification space 
$$TX:=\bigsqcup_{\widetilde{U}\in \mathcal U}(T\widetilde{U}/G)/\sim.$$

\begin{defn}
The  {\em tangent bundle}  of an $n$-dimensional orbifold $\mathcal X=(X, \mathcal U)$ is the  $2n$-dimensional orbifold 
$T\mathcal X=(TX, T\mathcal U)$ with the natural projection map $p \colon TX\to X$, which is a smooth map of orbifolds,
with fiber $p^{-1}(x)=T_{\widetilde{x}}\widetilde{U}/G_x$ for each $x\in X$.
\end{defn}

For a given local chart $(\widetilde{U}, G, \varphi)$ we choose a $G$-invariant inner product on $T\widetilde{U}$, and 
consider the corresponding frame manifold
$${\Fr}(\widetilde{U}):=\{(\widetilde{x}, B)\mid B\in O(T_{\widetilde{x}}\widetilde{U})\}$$
with the $G$-action defined by
$$g(\widetilde{x}, B)=(g\widetilde{x}, Dg_{\widetilde{x}}B).$$
Since $G$-action on $\widetilde{U}$ is effective, the $G$ action on ${\Fr}(\widetilde{U})$ is free, and 
the orbit space ${\Fr}(\widetilde{U})/G$ is a smooth manifold.
On the other hand, there is a right $O(n)$-action on ${\Fr}(\widetilde{U})/G$ which is induced from the group multiplication of $O(n)$,
and it can be seen that for each point $[(\widetilde{x}, A)]\in {\Fr}(\widetilde{U})/G$ its isotropy group is isomorphic to $G_x$.
Moreover by taking the quotient by the $O(n)$-action we have the  induced natural projection ${\Fr}(\widetilde{U})/G\to U$.

\begin{defn}\label{frame bundle}
The frame bundle of an orbifold $\mathcal X=(X, \mathcal U)$ is the space obtained by gluing the local charts
${\Fr}(\widetilde{U})/G\to U$ using the $O(n)$-transition functions obtained from the tangent bundle of $X$.
\end{defn}

One of the important properties of the frame bundle is the following theorem.

\begin{theorem}[Theorem 1.23, \cite{ALR}]\label{global quotient orbifold}
For a given effective $n$-dimensional  orbifold $\mathcal X$, its 
frame bundle ${\Fr}(\mathcal X)$ is a smooth
manifold with a smooth effective almost-free $O(n)$-action. The original orbifold $\mathcal X$ is naturally diffeomorphic to the
resulting quotient orbifold ${\Fr}(\mathcal X)/O(n)$.
\end{theorem}

\section{Torus orbifolds}\label{tororb}
The definition and basic properties of torus orbifolds are extensively discussed in \cite{HM}. 
In this section, we give the definition and the basic construction of locally standard torus orbifold.
All these are appropriate modifications of the quasitoric theory developed in \cite{DJ}, 
which also generalize the arguments of quasitoric orbifolds in \cite{PS}.

We begin by recalling the definition of manifold with corners from \cite[Section 6]{Da}. 
Various properties of manifold with corners and maps between them are studied in \cite{Jo}.

\begin{defn}
(a) A Hausdorff topological space $Q \subset \RR^m$ is called an $n$-dimensional
 \emph{manifold with corners} for $n\le m$ if any point $q \in Q$ has a neighborhood diffeomorphic to an open subset of the positive cone $\RR^n_{\geq 0}=
\{(x_1, \ldots, x_n)\in \RR^n\mid x_1 \geq 0, \ldots, x_n \geq 0\}$.
A $k$-dimensional face of $Q$  is defined in the natural way. 
Codimension-$1$ faces
of $Q$ are called \emph{facets}, and  faces of dimension 0 are called \emph{vertices}. 
We write $\mathcal F(Q)$ (resp. $\mathcal V(Q)$) for the set of facets (resp. vertices) of $Q$.
$Q$ is called \emph{smooth manifold with corners} if all the transition maps are smooth. 

(b) A manifold with corners is said to be  \emph{nice} if every codimension-$2$ face is a connected component of the  intersection of a unique collection of two facets.

(c) Two manifolds with corners are called \emph{diffeomorphic} if there is a diffeomorphism 
between them. 
\end{defn}

From the definition, it is easy to see that  any codimension-$k$ face of an $n$-dimensional nice manifold with corners $Q$ is a connected component of the intersection of a unique set of $k$ many facets of $Q$ for any  $0 \le  k \leq n$.
An $n$-dimensional {\em simple polytope} is a convex polytope each of whose vertices is the intersection of exactly $n$ facets. 
So simple polytopes are nice manifold with corners.  
In this article, we assume that every manifold with corners $Q$  is nice, orientable and smooth unless specifically mentioned otherwise.

For a $G$ space $X$ and an $H$ space $Y$, a map is said to be $\theta$-\emph{equivariant} for a homomorphism $\theta \colon G\to H$
if $f(gx)=\theta(g)f(x)$ for any $g\in G$ and $x\in X$. When we do not need to specify the homomorphism $\theta$, then $f$ is said to
be \emph{weakly equivariant} or {\em weak-equivariant}.

\subsection{Notations}\label{notation}
%Before we proceed, let us fix some notations.
Let $M$ be a free $\mathbb Z$-module of rank $n$, i.e., $M\cong \mathbb Z^n$.
Let
$$
\begin{array}{lll}
M_{\mathbb R}&:=&M\otimes_{\mathbb Z}\mathbb R \cong \mathbb R^n\\
T_M&:=&M_{\mathbb R} / M \cong (S^1)^n=T^n
\end{array}
$$
For a submodule $K$ of $M$ of rank $k$, let 
$$\widetilde{K}:=K_{\mathbb R}\cap M \cong \mathbb Z^k.$$ 
Then $K$ is submodule of 
$\widetilde{K}$ with finite index, and $\widetilde{K}$ is a rank $k$ submodule of $M$. Moreover the inclusion 
$\iota_K \colon K\hookrightarrow \widetilde{K}$ induces a surjective (covering) homomorphism 
$$\zeta_K \colon T_K\to T_{\widetilde{K}}$$ 
with finite abelian kernel $\ker(\zeta_K)\cong \widetilde{K} / K$  where $T_K=K_{\mathbb R} / K \cong T^k$ and $T_{\widetilde{K}}=\widetilde{K}_{\mathbb R }/ \widetilde{K} \cong T^k$. 
On the other hand, the inclusion  $\iota_{\widetilde{K}} \colon \widetilde{K}\hookrightarrow M$ induces an injective
homomorphism 
$$\zeta_{\widetilde{K}} \colon T_{\widetilde{K}}\to  T_M.$$
If $K$ is of rank $n$, in particular, then $\zeta_K \colon T_K\to T_{\widetilde K}=T_M$ is surjective with the finite abelian kernel.

\subsection{Definition of locally standard torus orbifold and its properties}\label{axdef}
Let $M$ be a free $\mathbb Z$-module of rank $n$, i.e., $M\cong \mathbb Z^n$
as above, or one may assume $M =\mathbb{Z}^n$.
Recall that the standard $T^n$ action on $\mathbb C^n$ is defined to be 
$$(t_1, \ldots, t_n)\cdot (z_1, \ldots, z_n)=(t_1z_1, \ldots, t_nz_n)$$ 
for $(t_1, \ldots, t_n)\in T^n$ and $(z_1, \ldots, z_n)\in \mathbb C^n$.
\begin{defn}\label{axideftororb}
A $2n$-dimensional  connected and closed effective orbifold $\mathcal{X}$ is called a \emph{locally standard torus orbifold} if the underlying topological space $X$ of $\mathcal{X}$ has an effective $T_M$ action such that for every point $x \in X$ 
there exist
\begin{itemize}
\item[P1)] a $T_M$-invariant neighborhood $ U$,
\item[P2)] a submodule  $N$ of $M$   of rank $n$ with the inclusion $\iota \colon N \to M$ 
and the induced surjective (covering) homomorphism $\zeta_N \colon T_N \to T_M$, and
\item[P3)] an orbifold chart $(\widetilde{U}, G, \varphi)$ over $U$  with $G = {\ker} \zeta_N $ where $\widetilde{U}$ is $\delta$-equivariantly diffeomorphic to an open set in 
$\CC^n$  for some isomorphism $\delta \colon  T_N \to T^n$,  and $\varphi \colon
\widetilde{U} \to U$ is a $\zeta_N$-equivariant map which induces an equivariant
 homeomorphism between $\widetilde{U}/G$ and $U$.

%\item[P4)] every $T^n$-invariant closed suborbifold $Y$ of $X$ of codemension $k$ has a $T^n$-invariant neighborhood which is $T^n$-equivariantly homeomorphic to a $k$-dimensional orbibundle on $Y$.   
\end{itemize}
\end{defn}
If the group $G$ in Definition \ref{axideftororb} is trivial for each $x\in X$,  then the orbifold $X$ is called a \emph{ locally standard torus manifold}. 

\begin{remark}\label{rmk of orbifold}
(1) From P3) each group $G = {\ker} \zeta_N \subset T_N$, hence
the singular set $\Sigma \mathcal X$ of the orbifold is contained in the singular part ${\Sing}(X, T^n)$ of the action
where ${\Sing}(X, T^n)$ is the union all singular $T^n$-orbits in $X$.

\noindent (2) Also from P3), the orbit space $Q=X/T^n$ is an $n$-dimensional
nice smooth  manifold with corners, and $\mathcal X$ is called a locally
standard torus orbifold over $Q$.
\end{remark}

%By an {\em orientable locally standard torus orbifold} $\mathcal{X}$ we
%assume both the orbifold $X$ and  the manifold with corners $Q$ to be orientable.

We also remark that even though  the definition of torus manifolds in \cite{HM}
assumes $X$ to have  fixed points,
we do not assume this here. In the case when $Q$ is a simple polytope and each group $G$ is trivial, $X$ is called a \emph{quasitoric manifold}, which  was introduced by Davis and Januszkiewicz \cite{DJ}
by the name of \emph{toric manifold}.  Also if $Q$ is a simple polytope but $G$'s are not necessarily trivial, then  
$X$ is defined to be  a  \emph{quasitoric orbifold} in \cite{PS}, and there the first author and M. Poddar studied several geometric and topological properties of them. 
%More general torus actions are discussed by Yoshida in \cite{Yo}.
An orbifold is called a $T^n$-\emph{orbifold} if there is an effective $T^n$-action on the 
underlying space.

\begin{example}\label{egs2n}
Consider the unit sphere
 $$S^{2n} = \{(z_1, \ldots, z_n, x) \in \CC^n \times \RR : |z_1|^2 + \cdots + |z_n|^2 + x^2 =1\}$$
with the following $T^n$-action on $S^{2n}$:
$$(t_1, \ldots, t_n, x) \cdot (z_1, \ldots, z_n, x) \mapsto (t_1 z_1, \ldots,  t_nz_n, x).$$ 
The points $(0, \ldots, 0, -1), (0, \ldots, 0, 1) \in S^{2n}$ are the fixed points of this action. Let $U_1 = S^{2n} -\{(0, \ldots, 0, -1)\}$ and $U_2=S^{2n}-\{(0, \ldots, 0,1)\}$.
One can show that $U_i$ is $\delta_i$-equivariantly diffeomorphic to $\CC^n$ with the standard action of $T^n$ for some $\delta_i \in \Aut(T^n)$ for $i=1, 2$. 
So $S^{2n}$ is a torus manifold. In particular, $S^{2n}$ is a locally standard torus orbifold with the orbit map $$\pi_s \colon S^{2n}\to Q^n_s=S^{2n}/T^n$$ where the
orbit space is given by 
$$Q^n_s = \{(x_1, \ldots, x_n, x) \in \RR^{n+1} ~:~ x_1^2 + \cdots + x_n^2 + x^2 = 1 ~~\mbox{and}~~ x_i \geq 0 ~~ \mbox{for} ~~ i=1, \ldots, n\}.$$ 
From the defining equations we get that $Q_s^n$ is a nice manifold with corners. 
\qed
\end{example}

Let $X$ be a $2n$-dimensional locally standard torus orbifold over $Q= X/T^n$ with the orbit map $\pi \colon X \to Q$.  Let $\mathcal F(Q)=\{F_1, \ldots, F_m\}$ denote the set of  facets of $Q$, and 
let $\overset{\circ}{F}_i$ denote the relative interior of $F_i$ .
By the local characterization of orbifold charts, the isotropy group
of any point $x \in \pi^{-1}(\overset{\circ}{F}_i) \subset X $ is a locally constant
circle subgroup of $T_M$. It is the image of a circle subgroup of $T_N$ under $\zeta_N$.
Thus it determines a locally constant vector $\lambda_i \in M$ up to sign, which is not necessarily a primitive vector of $M$. 
Since $\pi^{-1}(\overset{\circ}{F}_i)$ is connected, 
this vector $\lambda_i$ is uniquely determined up to sign for each facet $F_i$ of $Q$. 
This vector $\lambda_i$ is called the \emph{rational characteristic vector} of $F_i$. 
Thus we have the following  function:
\begin{equation}\label{r-ch-fun-of-X}
\lambda \colon \mathcal F(Q)\to M \cong \mathbb Z^n, \quad F_i\mapsto \lambda_i.
\end{equation}
In general, for an $n$-dimensional nice manifold with corners $P$ we define the following.
\begin{defn}\label{dichar}
A \emph{rational characteristic function} (or simply an \emph{r-characteristic function}) on an
$n$-dimensional manifold with corners  $P$  is a map $\xi \colon \mathcal{F}(P) \to \ZZ^n$ such
that whenever $F_{i_1} \cap \cdots \cap F_{i_k}\ne \emptyset$ the vectors $\xi(F_{i_1}), ...,  \xi(F_{i_k})$ are
linearly independent.
\end{defn}
By P3) of Definition \ref{axideftororb} the function $\lambda \colon \mathcal F(Q)\to \mathbb Z^n$ in (\ref{r-ch-fun-of-X})  is a rational characteristic function on $Q$.

On the other hand, since  every locally standard torus orbifold $X$ is compact, its orbit space $Q$ is a compact nice manifold with corners. So every facet of $Q$ has a collar neighborhood in $Q$, and hence the boundary $\partial{Q}$ has a collar neighborhood in $Q$. 
That is, the complement $Q_c$ of the union of suitable collar neighborhoods of all facets of $Q$ is diffeomorphic to $Q$ preserving the face structures. 
Note that $X_{c}=\pi^{-1}(Q_c)$ is the total space of a principal $T^n$-bundle, denoted by $\tau_c \colon E_{X_c} \to Q_c$. 
Since $Q_c$ is diffeomorphic to $Q$ as a manifold with corners, we may 
pull back $\tau_c$ to $Q$ to get a topological principal $T^n$ bundle
$$ \tau \colon E_X  \to Q$$ 
over $Q$ where $E_X$ is a nice manifold with corners and $\tau$ preserves the face structure.  
Thus, from a $2n$-dimensional locally standard torus orbifold $X$ with the orbit map 
$\pi \colon X \to Q$, we have obtained two data, a rational 
characteristic function
$\lambda \colon \mathcal F (Q)\to \mathbb Z^n$, and a principal $T^n$ bundle $\tau \colon E_X \to Q$.
We write these data by  $\{(Q, \lambda), (E_X, Q, \tau)\}$ and call it the \emph{combinatorial and topological data} of a locally standard torus 
orbifold $X$.

We remark that the definition of rational characteristic function is a slight generalization of the following well-known notion of characteristic function.
\begin{defn}\label{def_char_fun}
The function $\xi$ in Definition \ref{dichar} is called a \emph{characteristic function} if the vectors $\xi(F_{i_1}), ...,  \xi(F_{i_k})$ are a part of a basis of $\mathbb Z^n$ whenever $F_{i_1} \cap \cdots \cap F_{i_k}\ne \emptyset$.
\end{defn}

If $X$ is a locally standard torus manifold (instead of orbifold) over $Q$, then the r-characteristic function $\lambda$ in (\ref{r-ch-fun-of-X}) is indeed a characteristic function.
We finish this subsection with some examples of rational characteristic function.
\begin{example}\label{char_fun}
The manifold with corners in Figure \ref{egc3} (a) is obtained from  the rectangle $V_0V_1V_2V_3$ 
by deleting the interiors of the circle
$C$ and the triangle $V_4V_5V_6$. The manifold with
corners in Figure \ref{egc3} (b) is obtained from  the disk bounded by the  circle $C$ 
by deleting the interior of pentagon $V_0V_1V_2V_3V_4$.
Some r-characteristic functions of these $2$-dimensional manifolds
with corners are defined in the corresponding figure.
\begin{figure}[ht]
        \centerline{
           \scalebox{0.70}{
            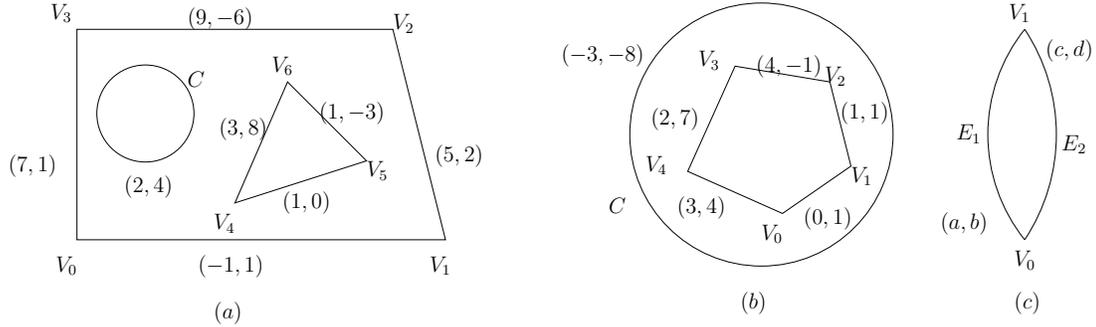
            }
          }
       \caption {Some r-characteristic functions.}
        \label{egc3}
      \end{figure}
      \qed
\end{example}

\begin{example}\label{eye-shape1}
The Figure \ref{egc3} (c) is an {\em eye-shape} $P^2$ with vertices $\{V_0, V_1\}$ and edges $\{E_0, E_1\}$. So an eye-shape is a nice manifold with corners. 
Define $\xi \colon \{E_0, E_1\} \to \ZZ^2 $~ by $$\xi(E_0)=(a,b) ~~ \mbox{, } ~~ \xi(E_1)=(c, d).$$ 
Then $\xi$ is an r-characteristic function if and only if $\{(a,b), (c, d)\}$ is a linearly independent set in $\ZZ^2$.  \qed
\end{example}

\subsection{The basic construction of locally standard torus orbifold}\label{condif}
Let $P$ be an $n$-dimensional nice manifold with corners equipped with a rational characteristic function
$\xi \colon \mathcal F(P)\to \mathbb Z^n$. Moreover, let a topological principal
$T^n$-bundle $\mu \colon E\to P$ over $P$ be given where $E$ is a nice
manifold with corners and $\mu$ preserves the face structure.
Let $\{(P,\xi),  (E, P, \mu)\}$ denote these data, which we call a \emph{combinatorial and topological data} over $P$.
In this subsection, we construct a $2n$-dimensional locally standard torus orbifold $X$ with the orbit space $P$ such that the corresponding combinatorial and topological data of $X$  is the given data  $\{(P,\xi),  (E, P, \mu)\}$.
Each point $x\in P$ lies in the relative interior $\overset{\circ}{F}$ of a unique codimension-$k$ face $F$ of $P$.  
If $k=0$ then $F=P$, and otherwise 
$F$ is a connected component of the intersection $F_{i_1}\cap\cdots\cap F_{i_k}$ of the unique collection
$\{F_{i_1}, \cdots, F_{i_k}\}\subset \mathcal F(P)$ because $P$ is nice.
Let $K(P)=0$, and let  $K(F)$ be the rank $k$ submodule of $M=\mathbb Z^n$ generated by the vectors $\xi(F_{i_1}),\ldots, \xi(F_{i_k})$.
Then as we have seen in Subsection~\ref{notation},  there is a surjective homomorphism
$\zeta_{K(F)} \colon T_{K(F)} \to T_{\widetilde{K}(F)}$ and an injective homomorphism  $\zeta_{\widetilde{K}(F)} \colon T_{\widetilde{K}(F)}\to  T_M$.
Let
\begin{eqnarray} 
G_F&:=& \widetilde{K}(F)/K(F)\label{def of GF},\\
T_F&:=&\img(\zeta_{\widetilde{K}(F)}\circ\zeta_{K(F)}).\label{def of TF}
\end{eqnarray} 
Then $G_F$ is a finite abelian group isomorphic to $\ker (\zeta_{K(F)})$, and $T_F$ is a rank $k$ torus subgroup of $T_M\cong T^n$.
We will adopt the convention that $T_P = 1$.

From the data $\{(P,\xi),  (E, P, \mu)\}$ we construct the space $X(P, \xi, \mu)$ as follow.
Define an equivalence relation $\sim$ on the total space $E$ of the principal bundle $\mu$ by
\begin{equation}\label{equ001}
  x \sim y ~ \mbox{if and only if}~ \mu(x) = \mu(y)~ \mbox{and}~ x=ty ~~ \mbox{for some} ~~ t \in T_F
\end{equation}
where $F$ is the face containing $\mu(x)=\mu(y)$ in its relative interior. The quotient space 
\begin{equation}\label{quotient construction}
X(P, \xi, \mu):=E/\sim
\end{equation}
has a natural $T^n$-action induced by the natural $T^n$-action on $E$. 
The orbit space
of $T^n$-action on $X(P, \xi, \mu)$ is diffeomorphic to $P$ as manifold with corners,  and the map
\begin{equation}\label{orbit map}
\pi \colon X(P, \xi, \mu) \to P, \quad \mbox{given by} \quad [x]^{\sim} \mapsto \mu(x)
\end{equation}
can be regarded as the orbit map, where $[x]^\sim$ denote the equivalence class of $x$.
In the case when $\mu$ is a trivial bundle,  we denote $X(P, \xi, \mu)$ by $X(P, \xi)$.

\begin{lemma}\label{def_const}
The space $X(P, \xi, \mu)$ is a locally standard torus orbifold over $P$ with the orbit map  $\pi$ in {\emph{(\ref{orbit map})}}.
\end{lemma}
\begin{proof}
 Let $[x]^{\sim} \in X(P, \xi, \mu)$. Hence $\mu(x) \in P$. We show that $[x]^{\sim}$ has a neighborhood which is $T^n$-equivariantly homeomorphic to $V_x/G_x$ where $V_x$ is a $T^n$-invariant open subset of $\CC^{n}$ with the standard $T^n$ action, and $G_x$ is a finite subgroup of $\Diff(V_x)$.% acting linearly on $V$.
 
First assume that $\mu(x)$ belongs to the interior  $\overset{\circ}{P}$. Then there is a neighborhood $U_x$ of $\mu(x)$ in $P$ which is diffeomorphic to an $n$-dimensional open ball in $\RR^{n}_{> 0}$. Clearly, 
$$\pi^{-1}(U_x) = T^n \times U_x/ \sim \, \cong T^n \times U_x$$
is a $T^n$-invariant neighborhood of $[x]^{\sim}$. Let $V_x = T^n \times U_x$, $G_x=\{1\}$ and $\rho_x \colon V_x \to X(P, \xi, \mu)$ be the inclusion. Then $(V_x, G_x, \rho_x)$ is an orbifold chart which satisfy the condition P3) of Definition \ref{axideftororb}.

Let $\mu(x)$ belong to the relative interior $\overset{\circ}{F}$ of a codimension-$k$ face $F$ of $P$ with $k>0$. 
Then there is a neighborhood $U_x$ of $\mu(x)$ in $P$ such that $U_x$ is diffeomorphic  to 
$ \RR^k_{\geq 0} \times \RR^{n-k}_{>0}$ as manifold with corners and $\mu^{-1}(U_x)$ is 
$T^n$-equivariantly homeomorphic  to $T^n \times U_x$ as manifold with corners. 
Let $$X_x :=\mu^{-1}(U_x)/\sim = \pi^{-1}(U_x).$$ 
Let $F$  be a component of the intersection $ F_{i_1} \cap \cdots\cap F_{i_k}$ of a unique collection of $k$ many facets of $P$. 
By Definition~\ref{dichar}, the set $\{\xi(F_{i_1}), \ldots, \xi(F_{i_k})\}$ is a linearly independent set of vectors of $\ZZ^n$. 
Let $K(F)$ be the submodule generated by  $\{\xi(F_{i_1}), \ldots, \xi(F_{i_k})\}$.

Let $K(F)^{\perp}:=\ZZ^n /\widetilde{K}(F)$, which is a free $\ZZ$-module of rank $n-k$. Then $\ZZ^n$ is isomorphic to $\widetilde{K}(F) \oplus K(F)^\perp$. Fixing an 
isomorphism of these $\ZZ$-modules we get $T^n = T_{\widetilde{K}(F)} \times T_{K(F)^{\perp}}$. On the other hand,  we also have an isomorphism $T^n \cong T_{K(F)} \times T_{K(F)^{\perp}}$ as $ T_{K(F)} $ and $ T_{K(F)^{\perp}}$ are $k$ and $n-k$ dimensional torus
respectively.

For  a face  $F^{\prime}$ containing $F$, let $K(F')$ be the submodule of $K(F)$ generated by the rational characteristic vectors corresponding to the face $F'$,  and 
let $T_{K(F')}$ be the corresponding torus group defined as before.
We now define an equivalence relation $\sim_F$ on $T_{K(F)} \times T_{K(F)^{\perp}} \times U_x $ by
\begin{equation}\label{toequiv}
 (t_1, s_1, q_1) \sim_F (t_2, s_2, q_2) ~
\mbox{if}~ q_1=q_2, s_1 =s_2 ~ \mbox{and} ~ t_{2}^{-1} t_1 \in T_{K(F')}
\end{equation}
 where $F^{\prime}$ is the unique face whose relative interior contains $q_1=q_2$. 
 
 Let 
 $$V_x = (T_{K(F)} \times T_{K(F)^{\perp}} \times U_x ) / \sim_F.$$
Using the aforementioned homeomorphism $U_x \to \RR^k_{\geq 0} \times \RR^{n-k}_{>0}$, we can get that $V_x$ is equivariantly homeomorphic to $\CC^{k} \times (\CC^*)^{n-k}$. Thus we have a commutative diagram of $T^n$-equivariant maps:
\begin{equation}
\begin{CD}
T_{K(F)} \times T_{K(F)^{\perp}} \times U_x  @>{\zeta_{K(F)} \times Id \times Id}>> T_{\widetilde{K}(F)} \times T_{K(F)^{\perp}} \times U_x @. \cong T^n \times U_x @. {\cong} \mu^{-1}(U_x)\\
@VVV  @VVV\\
 V_x @>{\zeta_x}>> X_x,
\end{CD}
\end{equation} 
where the vertical arrows are the quotient maps. Note that the map $\zeta_x$ is the orbit map of $G_F$-action on $V_x$ where $G_F$ is
as defined in (\ref{def of GF}).
Now the triple $(V_x, G_x, \zeta_x)$ with $G_x=G_F$ is an orbifold chart on $X_x$ which satisfies the condition P3) of Definition \ref{axideftororb}. 

Next we show that the orbifold charts $(V_x, G_x, \zeta_x)$'s are compatible
(see Definition \ref{def_compatibile_orbifold_chart}). The arguments are
essentially similar to \cite[Subsection 2.3]{GP2}, but few modifications
are needed. First we give coordinate structure on each $V_x$ in the following. Let
$$f_x \colon U_x \to \bb{R}^k_{\geq 0} \times \bb{R}^{n-k}$$ be the diffiomorphism as
manifold with corners. If we write $$f_x(u)=(f_{x,1}(u), \ldots, f_{x,k}(u),
f_{x, k+1}(u), \ldots, f_{x, n}(u))$$ then $f_{x, j }(u) \geq 0$, 
$f_{x, j}(u)=0$ if and only if $u \in U_x \cap F_{i_j} $ for $j=1, \ldots, k$, and
$f_{x, j}(u) > 0 $ for $j= k+1, \ldots, n$ ($k\geq 0$). Let $\{\bar{\xi}_{k+1},
\ldots, \bar{\xi}_{n}\}$ be a $\ZZ$-basis for $K(F)^{\perp} \subset \ZZ^n$.
When $F = Q$, we may consider $K(F) =0 \in \ZZ^n$. Then the kernel
of the map determined by the matrix $$\Lambda_x =[\xi_{i_1}^t, \ldots, \xi_{i_k}^t,
 \bar{\xi}_{k+1}^t, \ldots, \bar{\xi}_{n}^t]$$ is $G_F$. Let ${\bf \mathfrak{a}} = 
(\alpha_1, \ldots, \alpha_n)$ be the standard angular coordinates, and 
${\bf \mathfrak{a}}_x = (\alpha_{x,1}, \ldots, \alpha_{x,n})$ be the angular
coordinates of $T^n$ with respect to the basis  $\{\xi_{i_1}^t, \ldots,
\xi_{i_k}^t, \bar{\xi}_{k+1}^t, \ldots, \bar{\xi}_{n}^t\}$
of $~~~\ZZ^n \tensor_{\ZZ} \RR$. Then we get the following transformation 
between these angular coordinates $${\bf \mathfrak{a}} = \Lambda_x {\bf \mathfrak{a}}_x.$$
Let $u_{x,j} := f_{x, j}(u)\cos (2 \pi \alpha_{x, j})$
and $w_{x,j} := f_{x, j}(u)\sin (2 \pi \alpha_{x, j})$ for $j=1, \ldots, n$. 
We define $\phi_x \colon V_x \to \RR^{2n}$ by $$\phi_{x}([{\bf \mathfrak{a}}_x, u]^{\sim_F})
=(u_{x,1}, w_{x,1}, \ldots, u_{x,n}, w_{x,n}).$$
From the identification $\sim_{F}$ and condition on $f_x$ and using a
property of quotient map, we get that $\phi_x$ is homeomorphic onto its image
which is an open subset of $\RR^{2n}$. Then 
$(u_{x_1}, w_{x, 1}, \ldots, u_{x, n}, w_{x, n})$ can be taken as coordinate
structure on $V_x$.

The group action of $G_F$ on $V_x$ can be given by
similar relation as in \cite[(2.10)]{GP2}. Compatiblity of charts can be 
explained following the arguments in \cite[Subsection 2.3]{GP2}. Therefore $X(P, \xi, \mu)$ is
a locally standard torus orbifold.
\end{proof}

\begin{remark}\label{rem_tor_mf}
(1)  If  $P$ and the principal $T^n$ bundle $\mu$  in Lemma~\ref{def_const} are orientable, then so is $X(P,\xi,\mu)$.

(2) Observe that if $\xi$ satisfies the condition in Definition \ref{def_char_fun} then all local groups in the above orbifold charts are trivial. So in this case, 
$X(P, \xi, \mu)$ is a locally standard torus manifold.
\end{remark}

\begin{prop}\label{prop_eqi_diffeo}
Let $X$ and $ Y$ be locally standard torus orbifolds over $P$ and $Q$ respectively such that
the following diagram commutes
$$
\begin{CD}
X  @>{f}>> Y \\
@V{\pi_X}VV  @VV{\pi_Y}V\\
P @>{g}>> Q,
\end{CD}
$$
where $f$ is equvariantly homeomorphic, $\pi_X, \pi_Y$ are orbit maps,
and $g$ is a diffiemorphic as manifold with corners. Then $f$ is
equvariantly diffeomorphic.
\end{prop}

\begin{proof}
Using the coordinate description on each orbifold chart $(V_x, G_x, \zeta_x)$
in the proof of Lemma \ref{def_const} and modifying the arguments of the
proof of \cite[Lemma 2.3]{GP2}, one can complete the proof.
\end{proof}

\begin{example}\label{eye-shape2}
Let $P^2$ be an eye-shape and $\xi$ be an r-characteristic function on $P^2$ as in Example~\ref{eye-shape1}, and let $\mu$ be the trivial $T^2$-bundle over $P^2$. 
So $X(P^2, \xi)$ is a 4-dimensional orientable locally standard torus orbifold. 
Now we show that $X(P^2, \xi)$ is the orbit space of a finite group action on $S^4$. 
Let $Q^2_s$ be the manifold with corners as in Example~\ref{egs2n}.
Let $F_i = \{(x_1, x_2, x) \in Q^2_s ~\mid ~ x_i = 0\}$ for $i=1, 2$. 
Note that $\pi_s^{-1}(Q_s^2) =S^4\subset \mathbb C^2\times \mathbb R$ and  the isotropy group of the points in $\pi_s^{-1} (\overset{\circ}{F_i})$ is
$$
\begin{array}{lll}
Iso(\overset{\circ}{F_i})&=& \{(t_1, t_2)\in T^2 ~ \mid ~ (t_1z_1, t_2x_2, x) =(z_1, z_2, x)\}\\
				  &=&\begin{cases} S^1\times 1, &i=1\\
				                              1\times S^1, & i=2.
				        \end{cases}
\end{array}
$$				                            

Therefore the corresponding r-characteristic function is given by $$\lambda_s(F_1)=(1, 0)
~~\mbox{and}~~ \lambda_s(F_2)=(0, 1)$$ up to  choices of sign. Then one can show
that $(T^2 \times Q_s^2)/\sim ~ \cong S^4=\pi_s^{-1}(Q_s^2) $ where $\sim$ is
the equivalence relation defined in (\ref{equ001}). 
Let $f \colon Q_s^2 \to P^2$ be a diffeomorphism as manifold with corners such that $f(F_i)=E_i$ for $i=1, 2$. Consider the map $\ZZ^2 \to \ZZ^2$ determined by $$(1,0) \to (a,b) ~~\mbox{and}~~ (0,1) \to (c, d).$$ This induces a surjective Lie group homomorphism $\phi \colon T^2  \twoheadrightarrow T^2$, that is, $\phi$ is a finite covering homomorphism. From the definition of the equivalence relation $\sim$ it is clear that the map
 $$\phi \times \mbox{id} \colon T^{2} \times Q^2_s \to T^{2} \times P^2$$
induces a surjective map $$f_{\phi} \colon S^{4} \twoheadrightarrow X(P^2, \xi)$$
defined by $f_{\phi}([t,x]^{\sim}) = [{\phi}(t), f(x)]^{\sim}$ on the equivalence classes. 
The finite group $\ker{\phi}$ acts naturally on $S^{4}$. Since $\phi$ is a covering homomorphism with the finite covering group $\ker{\phi}$, the space $X(P^2, \xi)$ is 
diffeomorphic to the quotient space $S^{4}/\ker{\phi}$.
In particular, if $\xi$ is a characteristic function (see  Definition \ref{def_char_fun}), then $X(P^2, \xi)$ is $T^2$-weakly equivariantly diffeomorphic to $S^4$. \qed
\end{example}

\begin{example}\label{disc1}
Consider $S^3 =\{(z_1, z_2) \in \CC^2 : |z_1|^2 + |z_2|^2 = 1\}$ with an $S^1$-action on $S^3$ defined by 
\begin{equation}\label{s1_action_ons3}
\alpha \cdot (z_1, z_2) \to (z_1, \alpha z_2).
\end{equation}
 Then the orbit space  $P$ is a closed $2$-disc.
Let $\xi \colon \{\partial{P}\} \to \ZZ^2$ be the  map defined by $\xi(\partial{P}) = (a, b)$ where $(a, b)$ is a primitive vector in $\ZZ^2$. 
Then $\xi$ is a characteristic map on $P$.
So by Lemma \ref{def_const} and Remark \ref{rem_tor_mf}, $X(P, \xi)$ is an orientable locally standard torus manifold. Since $(a, b)$ is a primitive vector there is $(c, d) \in \ZZ^2$ such that $ad - bc =1$. So applying an automorphism of $T^2$ we may assume $(c, d) = (1, 0)$ and $(a, b) = (0, 1)$. Therefore we have the following
$$X(P, \xi) \cong (T^2 \times P)/\sim ~ \cong S^1 \times (S^1 \times P)/\sim ~ \cong S^1 \times S^3.$$ 
Here $T^2$ acts on $S^1\times S^3$ as follows: the first $S^1$ factor of $T^2$ acts on $S^1$ by left multiplication, and the second $S^1$ factor  acts on $S^3$ as in \eqref{s1_action_ons3}. Thus $X(P,\xi)$ is weakly equivariantly diffeomorphic to $S^1\times S^3$.
\begin{figure}[ht]
        \centerline{
           \scalebox{0.70}{
            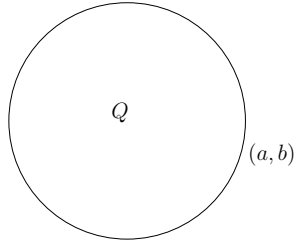
            }
          }
       \caption {A characteristic function on a closed disc.}
        \label{egc4}
      \end{figure}
\qed
\end{example}

The following theorem shows that any locally standard torus orbifolds can be constructed from the basic construction explained at the beginning of this subsection.

\begin{theorem}\label{clasi}
Let $X$ be a $2n$-dimensional locally standard torus orbifold over $Q$ with the associated   combinatorial and topological data $\{(Q, \lambda), (E_X, Q, \tau)\}$ as in Subsection \ref{axdef}. Let  $X(Q, \lambda, \tau)$ be the the locally standard torus orbifold
obtained  by the basic  construction from the data as in Subsection~\ref{condif}.
Then there  is a $T^n$-equivariant orbifold diffeomorphism from $X(Q, \lambda, \tau)$ to $X$ covering the identity on $Q$.
\end{theorem}
\begin{proof}
The proof is similar to the quasitoric orbifold case in \cite{PS}.
The basic idea of the proof is similar to that of Proposition 1.8 in \cite{DJ}, however there are two complications which require 
further arguments. The first one is that the orbit space $Q$ is not necessarily contractible, and the second one
is that we are dealing with orbifolds instead of manifolds.
The former complication can be fixed by considering principal $T^n$ bundle $\tau \colon E_X\to Q$ instead of the trivial $T^n$ bundle $T^n\times Q\to Q$.
So we need to take care of the latter complication. 
The main point is how to extend Lemma 1.4 of \cite{DJ} to the locally standard torus orbifold case, i.e., 
we need to construct a continuous map $$f \colon E_X\to X$$  which maps $\tau^{-1}(q)$ surjectively onto $\pi^{-1}(q)$ for each $q\in Q$, where $\pi \colon X\to Q$ is the orbit map.

Let ${\Fr}(X)$ be the frame bundle of the effective orbifold $X$. Then by Theorem~\ref{global quotient orbifold}
the frame bundle ${\Fr}(X)$ is a smooth manifold with  smooth effective almost-free $O(2n)$-action, and the quotient orbifold ${\Fr}(X)/O(2n)$ is diffeomorphic to $X$.
On the other hand,  since $X$ is a locally standard torus orbifold, $X$ has an effective
 $T^n$-action which induces an effective  $T^n$-action on ${\Fr}(X)$  commuting with
the above mentioned $O(2n)$-action.

We now apply the procedure for \lq\lq blowing up the singular strata" of the $T^n$-action on ${\Fr}(X)$ as explained in Lemma 1.4 of \cite{DJ}, to obtain an effective
$T^n$-manifold $\widetilde{\Fr}(X)$ with corners with only principal orbits. 
Indeed, $\widetilde{\Fr}(X)$ is obtained from ${\Fr}(X)$ by replacing each singular stratum by its normal sphere bundle in the order from the minimal stratum to the higher ones.
Then there is a natural map $$\widetilde{f} \colon \widetilde{\Fr}(X)\to {\Fr}(X)$$ which collapses each sphere bundle to the base points.
The $O(2n)$-action on ${\Fr}(X)$ also induces a free $O(2n)$-action on $\widetilde{\Fr}(X)$ because by Remark~\ref{rmk of orbifold}
the singular set $\Sigma X$ of the orbifold $X$ is contained in the singular part
 ${\Sing}(X, T^n)$ of the $T^n$-action on $X$.
Furthermore $T^n$ acts freely on $\widetilde{\Fr}(X)$ and commute with the action of $O(2n)$.
Let $\widetilde{X}:=\widetilde{\Fr}(X)/O(2n)$. Then $\widetilde{X}$ is a free $T^n$-space with the orbit space equal to $Q$.
Indeed, $\widetilde{X}$ is $T^n$-equivariantly homeomorphic to the total space $E_X$
of a principal bundle $\tau \colon E_X\to Q$.
The natural surjective map $\widetilde{f}$ induces a continuous $T^n$-equivariant
 map $$f \colon \widetilde{X} \cong E_X \to  {\Fr}(X)/O(2n)\cong X.$$
Since $O(2n)$-action commutes with $T^n$-action on $\widetilde{\Fr}(X)$ and ${\Fr}(X)$,
if $y \in f^{-1}(x) \subset E_X$ for some $x \in X$ then isotropy of $y$ is same
as $T_F$ where $F$ is the smallest face containing $\tau(y)$.
Therefore the map $f$ factor through the continuous map $E_X/ \sim ~ \to X$. 
 This is a $T^n$-equivariant  homeomorphism $X(Q, \lambda, \tau) \to X$ 
 covering the identity on $Q$. Then Proposition \ref{prop_eqi_diffeo}
 completes the proof. 
\end{proof}

We remark that blowing up of singular stratum may not be unique, but the
above procedure suffices our requirement.  

\begin{defn}\label{def_equiv}
Let $\{(P, \xi), (E, P, \mu)\}$ and $\{(P', \xi^{\prime}), (E^{\prime}, P', \mu^{\prime})\}$ be two combinatorial and topological data. 
They are called {\em equivalent} if there is a diffeomorphism
$\psi \colon P \to P'$ (as manifold with corners) and a $\delta\in\Aut(\ZZ^n)$ such that $\xi^{\prime}(\psi(F))= \pm\delta( \xi(F))$ for each $F \in \mathcal{F}(P)$ and $\mu$ is isomorphic to the pull back bundle $\psi^{\ast}(\mu^{\prime})$. 
\end{defn}

\begin{theorem}\label{thm_equiv}
Two locally standard torus orbifolds $X(P, \xi, \mu)$ and $X(P', \xi^{\prime}, \mu^{\prime})$ are $T^n$-weakly equivariantly diffeomorphic if and only if the corresponding characteristic and topological data $\{(P, \xi), (E, P, \mu))$ and $\{(P', \xi^{\prime}), (E^{\prime}, P', \mu^{\prime})\}$ are equivalent. 
\end{theorem}

\begin{proof}
Let $\Psi \colon X(P, \xi, \mu) \to X(P', \xi^{\prime}, \mu^{\prime})$ be a $T^n$-weakly equivariant diffeomorphism. Then it induces a diffeomorphism $\psi \colon P \to P'$ (as manifold with corners). This follows from the facts that $T^n$-action on each is
locally standard and the orbit map is smooth. Thus we have the following commutative diagrams, 
\begin{equation}\label{comp}
\begin{CD}
X(P, \xi, \mu) @>{\Psi}>> X(P', \xi^{\prime}, \mu^{\prime})\\
@VV{\pi}V  @VV{\pi^{\prime}}V\\
 P @>{\psi}>> P'
\end{CD}
\quad 
\mbox{and}
\quad
\begin{CD}
E @>{\Psi}>> E^{\prime}\\
@VV{\tau}V  @VV{\tau^{\prime}}V\\
 P_c @>{\psi}>> P'_c,
\end{CD}
\end{equation} 
where $\pi, \pi^{\prime}$ are orbit maps, 
$P_c$ (resp. $P'_c$) is the complement of a suitable  collor neighborhood of $\partial{P}$ (resp. $\partial{P'}$) in $P$ (resp. $P'$), $E=\pi^{-1}(P_c)$ and $E^{\prime} = (\pi^{\prime})^{-1}(P'_c)$.
Note that $P_c$ (resp. $P'_c$) is diffeomorphic to $P$ (resp. $P'$) as manifold with corners. Therefore $\xi^{\prime}(\psi(F))= \pm \delta( \xi(F))$  for a fixed $\delta\in Aut(\ZZ^n)$ and for all $F \in \mathcal{F}(P)$, and $\mu \cong \psi^{\ast}(\mu^{\prime})$.

Conversely, assume two  data $\{(P, \xi), (E, P, \mu)\}$ and $\{(P', \xi^{\prime}), (E^{\prime}, P', \mu^{\prime})\}$ are equivalent. So $\xi^{\prime}(\psi(F))= \pm \delta(\xi(F))$ for a fixed $\delta\in \Aut(\ZZ^n)$ and for all $F \in \mathcal{F}(P)$, and $\mu \cong \psi^{\ast}(\mu^{\prime})$ for some diffeomorphism $\psi \colon P\to P'$ as manifold with corners. So there is a bundle isomorphism  $\widetilde{\psi} \colon E \to E^{\prime}$ such that 
\begin{equation}
\begin{CD}
E @>{\widetilde{\psi}}>> E^{\prime}\\
@VV{\tau}V  @VV{\tau^{\prime}}V\\
 P_c @>{\psi}>> P'_c.
\end{CD}
\end{equation} 
Since $\widetilde{\psi}$ is $T^n$-equivariant and $\xi^{\prime}(\psi(F))= \pm \delta(\xi(F))$ for all $F \in \mathcal{F}(P)$, the map
$\widetilde{\psi}$ descends to a $T^n$-weakly equivariant map $\Psi \colon X(P, \xi, \mu) \to X(P, \xi^{\prime}, \mu^{\prime})$. 
Using the construction of $X(P, \xi, \mu)$ and $X(P, \xi^{\prime}, \mu^{\prime})$
 one can show that $\Psi$ is a homeomorphism. So by Proposition \ref{prop_eqi_diffeo}
 $\Psi$ is a diffeomorphism.
\end{proof}

We remark that the above two theorems are proved for the category of quasitoric manifolds in \cite{DJ} and for the category of quasitoric orbifolds in \cite{PS}. Also similar result are discussed for the category of $2$-torus manifolds in \cite{LM}. One can replace
diffeomorphism/diffeomorphic by homeomorphism/homeomorphic in Definition \ref{def_equiv}
and Theorem \ref{thm_equiv} and retain the conclusion of Theorem \ref{thm_equiv}
in homeomorphic category.

%%%%%%%%%%%%%

\subsection{Equivariant connected sum}\label{csum}
Here, we discuss connected sum of orientable locally standard torus manifolds along an orbit as in \cite{GK}. 
Let $A$ (resp. $B$) be an orbit of a $2n$-dimensional locally standard torus manifold $M$  (resp. $N$). 
Assume that the isotropy group of $A$ is isomorphic to that of $B$. 
So $A$ (resp. $B$) is a subset of a connected component of 
$M_{i_1} \cap \cdots \cap M_{i_\ell}$ (resp. $N_{i_1} \cap \cdots \cap N_{i_\ell}$) 
for a unique collection of characteristic submanifolds
$\{M_{i_1}, \ldots, M_{i_\ell}\}$ of $M$ (resp.  $\{N_{i_1}, \cdots, N_{i_\ell}\}$ of $N$), where 
a \emph{characteristic submanifold} of a locally standard torus manifold is the inverse image  of a facet by the orbit map.
Since $T^n$-action is locally standard, there are $T^n$-invariant small enough neighborhoods $U_A$ (resp. $U_B$)
of $A$ (resp. $B$) such that $U_A$ and $U_B$ are weak-equivariantly diffeomorphic to $\CC^\ell \times (\CC^{\ast})^{n-\ell}$. 
By changing the action of $T^n$ on $N$ by an automorphism of $T^n$ if necessary, we may assume that $T^n$-actions on
$U_A$ and $U_B$ are equivalent. That is, we may assume the isotropy group of $\overset{\circ}{M}_{i_j}$ is same as that of 
$\overset{\circ}{N}_{i_j}$ for $j = 1, \ldots, \ell$. By identifying the boundary of $M - U_A$ and $ N - U_B$ via an orientation
reversing equivariant diffeomorphism we get a manifold, denoted by $M \#_{A, B} N$, with a natural
locally standard $T^n$-action. So $M \#_{A, B} N$ is an oriented locally standard torus manifolds. 
For simplicity we denote the equivariant connected sum by $M \# N$ when the orbits $A$ and  $B$ are clear.
If $A, B$ are orbits of dimension $n$, then we say they are principal orbits and the space $M \# N$ is a connected sum along principal orbits.

Note that we can perform the equivariant connected sum construction for locally standard torus orbifolds along the orbits which belong to the smooth part of the orbifolds.

Let $Q$ be a nice 2-dimensional manifold with corners. 
So every component of $\partial Q$ is either boundary of a polygon, a circle, or an eye-shape  in Figure~\ref{egc3}(c). 
Note that $Q$ can be obtained from a closed surface $S_Q$, by removing the interior  of finitely many non-intersecting polygons,  
eye-shapes, or  discs. 
See Figure \ref{egc3} (a) for an example.

In the following lemma, let $S^4$ be the $T^2$-sphere in Example~\ref{egs2n}, and $S^1\times S^3$ be the $T^2$-manifold in
Example~\ref{disc1}.

\begin{lemma}\label{equi_classi}
Let $M$ be an orientable locally standard torus manifold (resp. orbifold) over a $2$-dimensional nice manifold  with corners  $Q$ such that $\partial{Q} \neq \emptyset$. 
Then $M$ is $T^2$-weakly equivariantly diffeomorphic to a connected sum of several copies  of 4-dimensional
quasitoric manifolds (resp. orbifolds), $T^2\times S_Q$, $S^4$   (resp.  $S^4/G$ as in Example \ref{eye-shape2}), and $S^1\times S^3$.

\end{lemma}

\begin{proof} We only prove the orbifold case. The manifold case is similar with $G$'s
are trivial groups. Let $\lambda$ be the r-characteristic function associated to $M$.
Since $Q$ is $2$-dimensional manifold with corners
with $\partial{Q} \neq \emptyset$, then $H^2(Q)=0$.
Hence  the principal $T^2$-bundle $\tau \colon E_M \to Q$ associated to $M$ is trivial.

Suppose $Q$ is obtained from $S_Q$ by removing  the interiors of copies of non-intersecting polygons $Q_{1_1}, \ldots, Q_{1_r}$, eye-shapes $Q_{2_1}, \ldots, Q_{2_s}$ and discs $Q_{3_1}, \ldots, Q_{3_t}$.  
For simplicity, we assume that  there is only one copy of $Q_i = Q_{i_1}$ for each $i=1,2,3$.
Note that  the facets of $Q$ is 
$\mathcal F(Q)=\bigcup_{i=1}^3\mathcal F(Q_i)$.
Now define  r-characteristic functions  $\lambda_i \colon \mathcal F(Q_i)\to \mathbb  Z^2$ on $Q_i$ 
to be the restriction $\lambda |_{\mathcal F(Q_i)}$  for each  $i=1,2,3$.

Let $$M^{\prime} = X(Q_1, \lambda_1) \# X(Q_2, \lambda_2) \# X(Q_3, \lambda_3) \# (T^2 \times S_Q)$$ be the connected sum of the following 4-dimensional locally standard torus orbifolds 
$$X(Q_1, \lambda_1), \quad X(Q_2, \lambda_2), \quad X(Q_3, \lambda_3) \quad \mbox{and} \quad  (T^2 \times S_Q)$$ where all  connected sums are performed along principal orbits. 
So $M^{\prime}$ is a locally standard torus manifold. 
Observe that the orbit space $Q^{\prime}$ of $M^{\prime}$ is  the connected sum of $Q_1, Q_2, Q_3$ and $S_Q$ at their respective interior points. 
So $Q^{\prime}$ is diffeomorphic as manifold with corners to $Q$. 
The r-characteristic function $\lambda^{\prime}$ associated to $M^{\prime}$ is induced by $\lambda_1, \lambda_2$ and $\lambda_3$. 
The r-characteristic functions on $Q$ and $Q^{\prime}$, and the trivial principal bundles on them satisfy the conditions of Definition \ref{def_equiv}. 
Therefore by Theorem \ref{thm_equiv},  $M \cong X(Q, \xi)$ is $T^2$-weakly equivariantly diffeomorphic to $X(Q^{\prime}, \lambda^{\prime}) \cong M^{\prime}$. 
By Example \ref{eye-shape2} and \ref{disc1}, the manifolds $X(Q_2, \lambda_2)$
respectively $X(Q_3, \lambda_3)$ are $T^2$-weakly equivariantly diffeomorphic to
$S^4/G$ respectively $S^1 \times S^3$. Also the manifold $X(Q_1, \lambda_1)$
 is a 4-dimensional quasitoric orbifold by \cite{PS}. This proves the  lemma.
\end{proof}

\subsection{Orbifold complex  projective space}\label{gwps}

A toric variety $X_{\Sigma}$ associated to a simplicial fan $\Sigma$ is called
a \emph{ toric orbifold}. The space $X_{\Sigma}$ is compact $2n$-dimensional toric
variety if and only if $\Sigma$ is complete fan in $\mathbb{R}^n$. It is well-known
that if one consider real torus $T^n \subset (C^*)^n$ action then it is a torus
orbifold. More studies on toric varieties can be found in \cite{Ful, CK}. 

\begin{defn}
Let $\Sigma$ be a complete simplicial fan in $\RR^n$ with $n+1$ many $1$-dimensional cones. The associated
toric orbifold $X_{\Sigma}$ is called an \emph{orbifold complex   projective space} of real dimension $2n$. 

\end{defn}
\begin{lemma}[Lemma 3.9, \cite{Sar3}]\label{cops}
 Let $X$ be a quasitoric orbifold over an $n$-dimensional simplex. Then $X$ is equivariantly diffeomorphic to an
 orbifold complex  projective space of real dimension $2n$.
\end{lemma}

We remark that a fake weighted projective space is a holomorphic generalization of weighted projective space, see \cite{Ka}.
A fake weighted projective space of real dimension $2n$  is determined by a complete simplicial fan generated by
$(n+1)$ many primitive vectors in $\ZZ^n$. So a fake weighted projective space
is an orbifold complex  projective space. Since the primitive vectors in $\ZZ$ are $-1 $ and $ 1$, the teardrop
$W\mathbb{P}(1, a)$ is not a fake weighted projective space if $a > 1$ but an orbifold
complex  projective space.

\subsection{Orbifold Hirzebruch surface}\label{orb_hir_sur}
A Hirzebruch surface is a nonsingular toric variety corresponding to a complete fan given by the Figure \ref{egc5} (A) where $b \in \ZZ$. Note that a Hirzebruch surface is a manifold. For more details on Hirzebruch surface see \cite{Hir, OR}. 
An \emph{ orbifold Hirzebruch surface} $X$ is defined to be a toric variety corresponding to a complete simplicial fan given by  
Figure~ \ref{egc5} (B) where $\{(a_i, b_i), (a_{i+1}, b_{i+1})\}$ are linearly independent vectors in $\ZZ^2$ for $i=1, \ldots, 4$ and $(a_1, b_1) = (a_5, b_5)$. 
One can show that $X$ with the restricted action of the compact torus $T^2\subset (\mathbb C^\ast)^2$ satisfies the condition of Definition \ref{axideftororb}. 
Therefore  orbifold Hirzebruch surfaces are quasitoric orbifolds.
\begin{figure}[ht]
        \centerline{
           \scalebox{0.70}{
            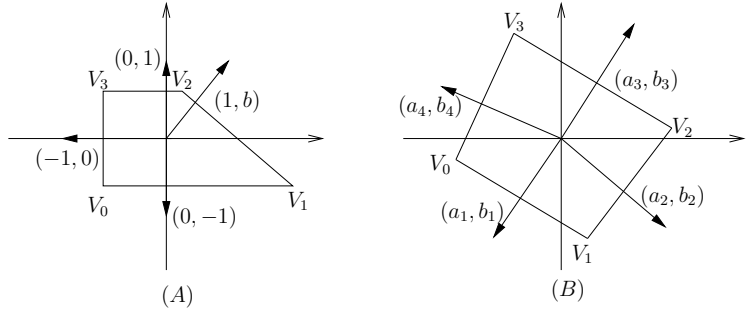
            }
          }
       \caption {Fans for Hirzebruch surface and  Hirzebruch orbifold.}
        \label{egc5}
      \end{figure}

Consider a rectangle $P^2$ with vertices $V_0, \ldots, V_3$ and edges
$V_0V_1, V_1V_2, V_2V_3, V_0V_3$. Define a map $$\xi \colon \{V_0V_1, V_1V_2,
V_2V_3, V_3V_0\} \to \ZZ^2$$ by $\xi(V_{i} V_{i+1}) = \epsilon_i (a_{i+1},
b_{i+1})$ where $i =0, \ldots, 3$, $V_4 = V_0$ and $\epsilon_i =\pm 1$. Then by
\cite[Lemma 2.2]{PS} and Theorem \ref{clasi} the  orbifold Hirzebruch surface
$X$ is $T^2$-equivariantly diffeomorphic to $X(P^2, \xi)$. 

%===========================================================

\section{Construction of orbifolds with boundary}\label{def}
From now on we assume all locally standard torus orbifolds and manifolds are orientable.
In this section we construct $(2n+1)$-dimensional orientable effective orbifolds  with $T^n$-actions whose boundaries are  locally standard torus orbifolds. To do this we need the notions of face-simple manifold with marked facets and rational super characteristic function defined on facets of it.

\begin{defn}
A \emph{face-simple manifold with marked facets}   is an $(n+1)$-dimensional oriented compact manifold with corners
$Y$ together with a subset $\{P_1, \ldots, P_m\}\subset\mathcal F(Y)$ of disjoint  facets of $Y$, called the \emph{marked facets}, such that 
%(XXXX here we change the notion exceptional faces to marked faces XXXX)
\begin{enumerate}
\item $Y$ is nice, i.e., any codimension-$k$ face of $Y$ is a connected component of the intersection of a unique set of $k$ many facets of $Y$ for any  $1 \le  k \leq n+1$, and
\item the vertex set  $\mathcal V(Y)$ is equal to  $ \bigsqcup_{i=1}^m  \mathcal V(P_i)$. 
\end{enumerate}
Face-simple manifold $Y$ with marked facets $P_1,\ldots, P_m$ is denoted by $Y[P_1,\ldots, P_m]$.
The facets   $\mathcal{F}(Y)\backslash \{P_1, \ldots, P_m\}$ is called the \emph{remaining facets} of $Y[P_1, \ldots P_m]$.
\end{defn}

Recall that the \emph{vertex-cut} $\VC(P)$ of a polytope $P$ is the polytope obtained from $P$ by cutting off disjoint  cone-shape neighborhoods $U_{v}$  of all vertices $v$ of $P$. 
For each vertex $v$ of $P$ let $F_v$ denote the facet of $\VC(P)$ corresponding to the intersection $\bar{U_v}\cap  \VC(P)$.

An $(n+1)$-dimensional polytope $P$ is said to be {\em edge-simple} if each one dimensional face of $P$ is the intersection of exactly $n$ facets of $P$, see \cite{Sar2}.
Clearly, the vertex cut $\VC(P)$  of an edge-simple polytope $P$ 
 is a face-simple manifold with marked  facets $\{F_v\mid  v\in \mathcal V(P)\}$. 
 
\begin{defn}\label{isofun}
Let $Y[  P_1, \ldots, P_m]$ be an $(n+1)$-dimensional face-simple manifold with marked facets,
and let  $\{F_1, \ldots, F_{m^{\prime}}\}$ be the remaining facets. 
A function 
$$ \eta \colon \{F_1, \ldots, F_{m^{\prime}}\} \to \ZZ^{n}$$
is called a \emph{rational super characteristic function }  (simply an \emph{rs-characteristic function}) on $Y[ P_1, \ldots, P_m]$ if the set of vectors $ \eta(F_{i_1}), \ldots,$ $ \eta(F_{i_{k}})$ are linearly independent in $\ZZ^{n}$ whenever $ F_{i_1}\cap \cdots \cap F_{i_k}\ne\emptyset$. 
The vector $ \eta(F_i)$ is called an \emph{rs-characteristic vector}  assigned to the facet $F_i$ for $i=1, \ldots, m^{\prime}$.
When the vectors $ \eta(F_{i_1}), \ldots,$ $ \eta(F_{i_{k}})$ are part of a basis of  $\ZZ^{n}$ 
whenever $ F_{i_1}\cap \cdots \cap F_{i_k} \ne\emptyset$, then $\eta$ is called  a \emph{super characteristic function} (simply an \emph{s-characteristic function}).
\end{defn}

Note that for r-characteristic function in Definition~\ref{dichar} the dimension
of manifold with corners and the rank of the target module are same, while
for rs-characteristic function the dimension of manifold with corners is larger
than the rank of the target module by $1$. Definition \ref{isofun} is a
generalization of the isotropy function of an edge-simple polytope given in \cite{Sar2}. 
In \cite{SS}, the authors considered hyper-characteristic functions, and in this case the dimension of the manifold is less than the rank of the target module by $1$.

Let $ \eta \colon \{F_1, \ldots, F_{m^{\prime}}\} \to \ZZ^{n}$ be an
rs-characteristic function on an $(n+1)$-dimensional face-simple manifold
with marked facets  $Y[  P_1, \ldots, P_m]$. For each marked facet $P_j$
let $\mathcal F(P_j):=\{G_{j_1}, \ldots, G_{j_{h_j}}\}$ be its facets.
Then  for each facet $G_{j_i}$ of $P_j$, there exists a unique facet
$F_{j_i}$  among the remaining facets of $Y[  P_1, \ldots, P_m]$ such that
$G_{j_i}=P_j\cap F_{j_i}$. We now define  functions
\begin{equation} \label{hr to r}
\xi_j \colon \{G_{j_1}, \ldots, G_{j_{h_j}}\} \to \ZZ^n \quad \mbox{by} \quad \xi_j(G_{j_i}) = \eta({F_{j_i}}).
\end{equation}
for $j=1,\ldots, m$.
Then the following lemma is obvious from the definition of rs-characteristic function.

\begin{lemma}\label{lemma:hr to r}
The function $\xi_j$ defined in (\ref{hr to r}) is an r-characteristic function on $P_j$ for $j=1,\ldots, m$. 
\end{lemma}

Let  $\mu \colon E\to Y$  be an orientable principal $T^n$-bundle over $Y$, and 
let $\eta$ be an rs-characteristic function on $Y[P_1,\ldots,P_m]$. Then the collection
$\{(Y[P_1,\ldots,P_m], \eta), (E, Y, \mu)\}$ is called   
a {\em combinatorial and topological data} on $Y[P_1,\ldots,P_m]$. 
From this data we construct a $(2n+1)$-dimensional $T^n$-orbifold with boundary where the boundary is a disjoint union of some locally standard $2n$-dimensional orbifolds.
The construction is quite similar to the basic construction  in Subsection~\ref{condif}.

Let $F$ be a codimension-$k$ face of $Y[P_1, \ldots, P_m]$ for $0<k\le n+1$. If $F$ is a face of one of the marked facets $P_i$ for some $i\in \{1,\ldots, m\}$, then there is a unique collection of remaining facets $F_{i_1},\ldots, F_{i_{k-1}}$ such that $F$ is a component of
the intersection $F_{i_1}\cap\cdots\cap F_{i_{k-1}}\cap P_i$. Otherwise, there is a unique collection of remaining facets $F_{i_1},\ldots, F_{i_{k}}$
such that $F$ is a component of the intersection $F_{i_1}\cap\cdots\cap F_{i_{k}}$.
Let
\begin{equation}%\label{charp}
K(F) = 
	\left\{ \begin{array}{ll} \< \eta(F_{i_j}) \mid j = 1, \ldots, \ell-1 \>  & \mbox{if} ~ F \subseteq  F_{i_1} \cap \cdots \cap F_{i_{\ell-1}} 	\cap P_i,\\
	\< \eta(F_{i_j}) \mid j = 1, \ldots, \ell \>  & \mbox{if} ~ F \subseteq F_{i_1} \cap \cdots \cap F_{i_{\ell-1}}\cap F_{i_\ell}
	\end{array} \right.
\end{equation}
where $\<\alpha_i : i=1, \ldots, s\>$ denotes the submodule  of $\ZZ^n$
 generated by the vectors $\alpha_{i} $ for $i = 1, \ldots, s$.
We define an equivalence relation $\sim_b$ on the total space $ E $ of the principal $T^n$-bundle $\mu \colon E \to Y$ as follows: for $x, y\in E$ 
 \begin{equation}\label{equilam}
 x \sim_b y \quad \mbox{if and only if} \quad  \mu(x) = \mu(y) ~~ \mbox{and} ~~ x=uy ~~\mbox{for some}~~ u \in T_F
\end{equation} 
where $ F $ is the unique face of $ Y $ containing $ \mu(x) = \mu(y) $ in its relative interior and, $T_F$ is the subgroup of $T^n$ as defined in
(\ref{def of TF}).  Let 
$$W(Y, \eta, \mu):= E/ \sim_b$$ be the equivalence classes. Then $W(Y, \eta, \mu)$
 has a  $T^n$-action induced from the action of $T^n$ on $E$.
Let $[x]^{\sim_b}$ denote the equivalence class  of $x$ in $W(Y, \eta, \mu)$, and
let $\pi \colon W(Y, \eta, \mu) \to Y$ be the projection map defined by $\pi([x]^{\sim_b})=\mu(x)$.

\begin{theorem}\label{orbbd}
Let $Y[P_1,\ldots,P_m]$ be an $(n+1)$-dimensional face-simple manifolds with
 marked facets, let $\eta$ be an rs-characteristic function on $Y[P_1,\ldots,
 P_m]$, and let $\mu$ be an orientable principal $T^n$-bundle over $Y$.
Then $ W(Y, \eta, \mu)$ is a $(2n+1)$-dimensional orientable effective
$T^n$-orbifold with the  boundary consisting of
$m$ many disjoint orientable locally standard torus orbifolds.
\end{theorem}

\begin{proof}
For any $[x]^{\sim_b} \in W(Y, \eta, \mu)$ we show that there exists a neighborhood of $[x]^{\sim_b}$ which is $T^n$-equivariantly homeomorphic to $V/G$ where $V$ is a $T^n$-invariant open subset of $\CC^{n} \times \RR_{\geq 0}$ on which 
$T^n$ action on $\CC^n$ is standard and  on  $\RR_{\ge 0}$ is trivial,
and $G$ is a finite subgroup of $\Diff(V)$.
There are three cases:
\begin{enumerate}
\item when $\mu(x)\in \overset{\circ}{Y}$,
\item when $\mu(x)\in \overset{\circ}{F}$  where $F$ is a face of $Y$ which is not contained  in any of the marked facets $P_i$, and
\item when $\mu(x)\in P_j$ for some $j \in \{1, \ldots, m\}$.
\end{enumerate}

The proofs of the first two cases  are almost identical to  that of Lemma~\ref{def_const}. 
The only difference is that in Lemma~\ref{def_const} we found a $T^n$-invariant open subset $V$  in $\CC^n$ instead of
$\CC^{n} \times \RR_{\geq 0}$, but this difference does not cause any difficulty here.

So we prove the third case.
Now assume  $\mu(x)\in P_j$ for some $j=1, \ldots, m$.
Since $Y$ is a compact manifold with corners, there exists a collar neighborhood $U_{x}$ of $P_j$ in $Y$. 
So there is a diffeomorphism $g_x \colon U_{x} \to P_j \times [0,1)$ as manifold with corners and 
a $T^n$-equivariant homeomorphism $ \widetilde{g}_x \colon \mu^{-1}(U_x) \to \mu^{-1}(P_j) \times [0, 1)$ 
such that the following diagram commutes. 
\begin{equation}
\begin{CD}
\mu^{-1}(U_x) @>{\widetilde{g}_x}>> \mu^{-1}(P_j) \times [0, 1)\\
@VVV  @VVV\\
 U_x @>{g_x}>> P_j \times [0, 1).
\end{CD}
\end{equation} 
Observe that the equivalence relation $\sim_b$ in \eqref{equilam} does not affect the second component of $\mu^{-1}(P_j) \times [0, 1)$ and $\sim_b$ is same as the relation $\sim$ in \eqref{equ001} on $\mu^{-1}(P_j)$. Let $\{G_{j_1}, \ldots, G_{j_{h_j}}\}$ be the facets of $P_j$. 
Then $G_{j_i} = P_j \cap F_{j_i}$ for a unique remaining  facet $F_{j_i}$ of $Y[P_1, \ldots, P_m]$. 
Define $$\xi_j(G_{j_i}) = \eta(F_{j_i}) \quad \mbox \quad i=1, \ldots, h_j.$$ 
Then $\xi_i$ is an r-characteristic function on $P_j$ by Lemma~\ref{lemma:hr to r}. 
Let $E_j = \mu^{-1}(P_j)$ and $~\mu_j \colon E_j \to P_j$ be the restriction of $\mu$. 
Then $\mu_j \colon E_j \to P_j$ is a principal $T^n$ bundle. 
So we have the following $T^n$-equivariant homeomorphisms.
\begin{equation}\label{eq002}
\mu^{-1}(U_x)/ \sim_b \cong (\mu^{-1}(P_j)/ \sim_b) \times [0, 1) \cong (\mu^{-1}(P_j)/ \sim) \times [0, 1)
\cong X(P_j, \xi_j, \mu_j) \times [0, 1)
\end{equation}
 where $X(P_j, \xi_j, \mu_j)$ is the orientable locally standard torus orbifold associated
to the characteristic and topological data $\{(P_j, \xi_j), (E_j, P_j, \mu_j)\}$. 
Let $y^{\prime} \in  X(P_j, \xi_j, \mu_j) \times [0, 1)$ be the image of $[x]^{\sim_b}$ under the homeomorphism in \eqref{eq002}. 
Since $\mu(x) \in P_j$, we have  $y^{\prime} = (y, 0)$ for some $y \in X(P_j, \xi_j, \mu_j)$. So there is an orbifold chart $(V, G, \psi)$ on a neighborhood of $y$ in $X(P_j, \xi_j, \mu_j)$. Thus $(V \times [0, 1), G, \psi \times id)$ is an orbifold chart  on a neighborhood of $[x]^{\sim_b}$ in $W(Y, \eta, \mu)$. One can check the compatibility of these charts
using the similar arguments as in the proof of Lemma \ref{def_const}. 

Hence the space $W(Y, \eta, \mu)$ is an orientable compact effective orbifold whose  boundary is the disjoint union of 
orientable locally standard torus orbifolds $\{X(P_j, \xi_j, \mu_j)$ for $ j = 1, \ldots, m\}$. 
Recall that by our assumption $Y$ is orientable. So choosing an orientation of $Y$ and $T^n$ we get an
orientations on $W(Y, \eta, \mu)$ and the boundary orbifolds.
\end{proof}

When the principal $T^n$-bundles $\mu$ is trivial, we write $W(Y, \eta)$ instead of $W(Y, \eta, \mu)$. 

\begin{example}\label{egchar2}
Two rs-characteristic functions $\eta^i$ of the face-simple manifolds 
 with marked facets  $Y_i$  for $i=1,2$ are given in Figure \ref{egc2}. 
(This figure is same as \cite[Figure 4]{Sar2}, but the functions on them are different.) 
Here principal $T^2$-bundles $\mu^i$ on $Y_i$ are trivial, since $Y_i$ is contractible for $i=1, 2$. 
In $(a)$ all marked facets $P_1, P_2, P_3, P_4$ are triangles, and hence $W(Y_1, \eta^1)$ is an orientable orbifold whose boundary consists of  locally standard orbifolds $X(P_j, \xi_j)$ which are orbifold complex 
projective spaces defined in Subsection~\ref{gwps} for $j\in \{1,2,3,4\}$. 

In $(b)$, the marked facets $P_1,P_{2}, P_{3}, P_{4}$ are triangles and $P_{5}$ is a rectangle. Thus $W(Y_2, \eta^2)$ has the boundary 
consisting of orbifold complex projective spaces  for $j\in \{1,2,3,4\}$ and an orbifold Hirzebruch surface   $X(P_5, \xi^2_5)$ defined in Subsection~\ref{orb_hir_sur}.
\begin{figure}[ht]
        \centerline{
           \scalebox{0.70}{
            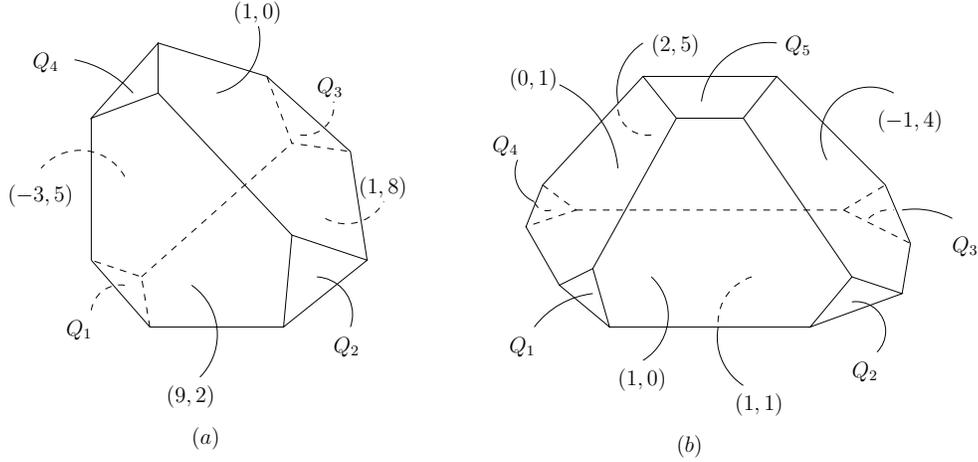
            }
          }
       \caption {Some rs-characteristic functions of face-simple manifolds $Y_1$ and $Y_2$ with marked facets.}
        \label{egc2}
      \end{figure}
      \qed
\end{example}

If $\eta$ satisfies that the set of vectors $ \{\eta(F_{i_1}), \ldots,$ $ \eta(F_{i_{k}})\}$
is a part of a basis of $\ZZ^{n}$ whenever the intersection of the facets $ \{F_{i_1}, \ldots,$ $ F_{i_k}\}$
is nonempty, then all the local groups in the proof of Lemma \ref{orbbd} are trivial. So, in this case, we have the following corollary.
\begin{corollary}\label{manbdd}
Under the same assumption as in the last paragraph, the space $ W(Y, \eta, \mu) $ is a $(2n+1)$-dimensional smooth orientable bounded $T^n$-manifold whose boundary is 
a disjoint union of orientable locally standard torus manifolds.
\end{corollary}
We remark that the above corollary is a generalization of \cite[Lemma 4.4]{Sar3}.

%%%%%%%%%%%%%%%%%%%%%%%%%%%%%%%%%%%%%%%%%%%%%%%%%%%%%%%%%%%%%%%%%%%%%

\section{Equivariant cobordism of torus orbifolds}\label{cobor}
In this section we exhibit several explicit  cobordisms among orientable locally standard  torus orbifolds.
First, we recall the definition of equivariant  cobordism of  orientable locally standard torus orbifolds.

\begin{defn}\label{def_equi_cob}
Two $2n$-dimensional orientable locally standard torus orbifolds $X_1$ and $X_2$ are said to be {\em equivariantly cobordant}  (or {\em torus cobordant}) if there exists a $(2n+1)$-dimensional orientable effective $T^n$-orbifold $W$ with boundary $\partial W $ such that $\partial W$ is $T^n$-equivariantly diffeomorphic to $ X_1 \sqcup (-X_2)$ under an
orientation preserving diffeomorphism. Here $-X_2$ denotes $X_2$ with the opposite orientation.
\end{defn}

The above cobordism relation has transitive property. It follows from the fact that
torus orbifolds considered here are compact, so one can construct a tubular neighborhood
of a boundary component. Let $\mathcal{OC}_{n}$ be the group of equivariant cobordism
classes of $2n$-dimensional orientable locally standard torus orbifolds, 
where the group structure is given by the disjoint union.

\begin{example}\label{bdd_lens}
Any torus orbifold over an eye-shape in Example~\ref{eye-shape1} is $T^2$-equivariantly cobordant to zero, i.e.,  $T^2$-equivariantly  a boundary. 
%We adhere to the notations of Examples \ref{egs2n}, \ref{eye-shape1} and \ref{eye-shape2}.
%In this example we show that any torus orbifold over an eye-shape is $T^2$-equivariantly cobordant to zero, i.e., a $T^2$-equivariantly boundary. 
Indeed, let $\rho \colon \ker{\phi} \times S^4 \to S^4$ be the action discussed in Example \ref{eye-shape2}. 
Then we have the following commutative diagram.
\begin{equation}
\begin{CD}
\ker{\phi} \times S^{4} \times I @>{\rho \times \mbox{id}}>> S^{4} \times I\\
@VVV  @VVV\\
 \ker{\phi} \times (S^4 \times I)/ (S^4 \times \{0\}) @>>> (S^4 \times I)/ (S^4 \times \{0\}) .
\end{CD}
\end{equation}
Since $(S^4 \times I)/ (S^4 \times \{0\}) \cong D^5$, the space $D^5/\ker{\phi}$ is an orientable orbifold with boundary $S^4/\ker{\phi}$. 
That is, the torus orbifold $X(P^2, \xi)= S^4/\ker{\phi}$ over an eye-shape $P^2$ is $T^2$-equivariantly a boundary. 
So the cobordism class $[X(P^2, \xi)] = 0$ in the group $\mathcal{OC}_{4}$. \qed
\end{example}

Before we give equivariant cobordism results on orientable locally standard torus orbifolds, 
let us give a simple result on  locally standard torus manifolds over polytopes with simple holes. 
Polytopes with simple holes are defined in  \cite{PS2}  as follows. 
Let $Q_0$ be an $n$-dimensional simple  polytope in $\RR^n$. Let $ Q_1,
Q_2, \ldots, Q_\ell $ be a collection of disjoint simple polytopes in the interior of $Q_0$. 
Let $Q = Q_0 - \bigcup_{k=1}^{\ell} \overset{\circ}{Q}_k$.
This $Q$ is called an $n$-dimensional {\em polytope with simple holes} $Q_1, \ldots, Q_\ell$. Note that $Q$ is a nice manifold with corners. 

Let $M$ be an orientable locally standard torus manifold over the polytope $Q$ with $\ell$ simple holes $Q_1, \ldots, Q_\ell$. 
Let $\lambda$ be the characteristic function associated to $M$, and let $\lambda^i$ be the restriction of $\lambda$ to $Q_i$ for $i=0, 1, \ldots, \ell$. 
\begin{lemma}\label{equi}
Let $M$ be as above.  If $H^2(Q)=0$, then $M$ is equivariantly cobordant to the disjoint union $M(Q_0, \lambda^0) \sqcup
\cdots \sqcup M(Q_\ell, \lambda^\ell)$ by a $(2n+1)$-dimensional $T^n$-manifold. 
\end{lemma}
\begin{proof}
Since $H^2(Q)=0$, any principal $T^n$-bundle over $Q$ is trivial. 
Hence by Theorem~\ref{clasi}, $M$ is equivariantly diffeomorphic to $M(Q, \lambda)$.
Following the proof of \cite[Lemma 2.1]{PS2}, one can show that  $M(Q, \xi)$ is $T^n$-equivariantly diffeomorphic to the connected sum
$M(Q_0, \lambda^0)\sharp M(Q_1, \lambda^1) \sharp\cdots\sharp M(Q_\ell, \lambda^\ell)$ where the connected sum occurred along a principal orbit of quasitoric manifolds $M(Q_i, \lambda^i)$ for $i=0, \ldots, \ell$.

%(XXX I can not find the exact place of the mentioned result from  Subsection 2.3 of \cite{PS2}. XXXXX)

On the other hand, it is well-known that equivariant connected sum of two manifolds are equivariantly cobordant to the disjoint union of them. Therefore we have the lemma.
\end{proof}

For a positive integer $k$, let $[k]$ denote the set $\{1, \ldots, k\}$. 
For an integer $1 \le j \leq k$, let $\mathcal P_j[k]$ denote the collection of subsets of $[k]$ with $j$ elements.

\begin{lemma}\label{indl}
Let $j$ and $k$ be positive integers such that $j \leq k$. 
Let  a subset $\mathcal A \subset \mathcal P_j[k]$ and a set $\{\xi_1, \ldots, \xi_k\}$ of vectors in $\ZZ^j$ be given such that 
for any $\{i_1, \ldots, i_j\} \in \mathcal A$ the vectors $\xi_{i_1}, \ldots, \xi_{i_j}$ are linearly independent. 
Then there exists a primitive vector $\xi_0 \in \ZZ^j$ such that the set $\{ \xi_0, \xi_{i_1}, \ldots, \xi_{i_{\ell-1}}, \widehat{\xi_{i_\ell}}, \xi_{i_{\ell+1}}, \ldots, \xi_{i_{j}}\}$ is a linearly independent
set of vectors in $\ZZ^j$ for any $\{i_1, \ldots, i_j\} \in \mathcal A$ and $\ell= 1, \ldots, j$. Here  $\ \widehat{}\ $
represents the omission of the corresponding entry.
\end{lemma}

\begin{proof}
The $j=1$ case follows because $|\ZZ|=\infty$. 
Let $j > 1$. For each $I=\{i_1,\ldots, i_j\} \in \mathcal A$ and for each $i_\ell\in I$, let $K_{I, {i_\ell}}$ be the rank $(j-1)$ submodule of $\ZZ^{j}$
generated by $\{  \xi_{i_1}, \ldots, \xi_{i_{\ell-1}}, \widehat{\xi_{i_\ell}}, \xi_{i_{\ell+1}}, \ldots, \xi_{i_{j}}\}$. 
Then $\bigcup_{\ell \in I,\, I\in \mathcal A} K_{I, {i_\ell}}$ is a finite union of rank $(j-1)$ submodules of $\ZZ^j$. 
Therefore there exists a nonzero vector 
$$\xi_0\in \ZZ^j - \bigcup_{\ell \in I, \, I\in \mathcal A} K_{I, {i_\ell}}.$$
If we choose $\xi_{0}$ to be a primitive vector, then it clearly satisfies the conditions in the lemma.
\end{proof}

Now we prove one of the main theorems of this article.
\begin{theorem}\label{coborthm}
Let $X$ be an orientable $2n$-dimensional locally standard torus orbifold with $k$ fixed points. Then $X$ is
equivariantly cobordant to a disjoint union of $k$ many  orbifold complex projective spaces.  
\end{theorem}
\begin{proof}
Let $X$ be a $2n$-dimensional orientable locally standard torus orbifold with $k$ fixed points. So the orbit
space $Q := X/ T^n$ is an orientable smooth compact manifold with corners with $k$ vertices.
Let $\{(Q, \lambda'), (E_X, Q, \tau')\}$ be the combinatorial and topological data of $X$ as obtained in Subsection~\ref{axdef}, 
where $\lambda'$ is an r-characteristic function on $Q$ and $\tau' \colon E_X \to Q$ is an orientable principal $T^n$-bundle. 
Then by Theorem~\ref{clasi}, $X$ is equivariantly diffeomorphic to $X(Q, \lambda', \tau')$.
Suppose $Q$ is a subset of $\RR^{\ell -1}$ for some $\ell \geq n+1$. 
Let $\mathcal F(Q):=\{F_1, \ldots, F_{m}\}$ be the facets of $Q$, and let $\mathcal V(Q):=\{V_1, \ldots, V_k\}$ be the vertices of $Q$. 
Let $Y := Q \times \bigtriangleup^1 \subset \RR^{\ell}$ where  $\bigtriangleup^1=[0,1]$ is the $1$-simplex. 
Then $Y$ is an $(n+1)$-dimensional nice manifold with corners, whose facets are 
$$\mathcal F(Q\times \Delta^1)=\{F_1 \times \bigtriangleup^1, \ldots, F_{m} \times \bigtriangleup^1, Q \times \{0\}, Q \times \{1\}\}$$ 
where $\{0\}$ and $\{1\}$ are the vertices of $\bigtriangleup^1$.
Let $V_i^0 = V_i \times \{0\}$ for $i = 1, \ldots, k$. 
We perform vertex-cut of $Y$ around the vertices $V_i^0$ for  $i = 1, \ldots, k$ as follows.
Consider an open ball $B^{\ell}_i$ in $\RR^{\ell}$ around each vertex $ V_i^0$  of $Y \subset \RR^{\ell}$ such that the followings are satisfied.
\begin{enumerate}
\item $\bar{B^{\ell}_i} \cap Y$ is diffeomorphic as manifold with corners to the $(n+1)$-simplex in $\RR^{n+1}$ for each $i=1, \ldots, k$, where $\bar{B^{\ell}_i}$ denotes the closure of $B^{\ell}_i$ in $\RR^{\ell+1}$.
\item $\bar{B^{\ell}_i} \cap \bar{B^{\ell}_j} \cap Y=\emptyset$  when $i \neq j$.
\item $\bar{B_i^{\ell}} \cap Q \times \{1\}=\emptyset$ for all $i= 1, \ldots, k$.
\end{enumerate}
Then $Y_Q:=Y-\bigcup_{i=1}^k B_i^{\ell}$ is the desired vertex-cut, which is an 
orientable smooth nice  $(n+1)$-dimensional manifold with corners.
Let $S^{\ell}_i=\partial(\bar{B_i^{\ell}})$ for $i=1, \ldots, k$. 
Then $Q_{i} := S^{\ell}_i \cap Y$ is the  facet of $Y_Q$  corresponding to the vertex $V_i^{0}$ for each $i= 1, \ldots, k$. 
Since each vertex of $Q_i$ is an interior point of an edge of $Y$, and $Y$ is an $(n+1)$-dimensional nice manifold with corners, 
$Q_i$ is diffeomorphic as manifold with corners to the $n$-simplex for each $i= 1, 2, \ldots, k$. 
Let 
$$
\begin{array}{lcl}
	F^0_j	& 	=	&		F_j \times \bigtriangleup^1 - \bigcup_{i=1}^k B_i^{\ell} \quad \textrm{for } j= 1, \ldots, m,\\
	F^0_{m+1}&	= &	Q \times \{0\}- \bigcup_{i=1}^k B_i^{\ell},\\
	Q_{k+1} &=& 	Q \times \{1\}. 
	\end{array}
$$

%$$Q_{k+1} = Q \times \{1\},~~ F^{\prime}_{m+1}= Q \times \{0\}- \{\cup_1^k B_i^{\ell} \cap Y\} ~~\mbox{and} ~~ 
%F^{\prime}_i = F_i \times \bigtriangleup^1 - \{\cup_1^k B_i^{\ell} \cap Y\}$$ 
%for $i= 1, \ldots, m$. 
\noindent So the facets and the vertices of $Y_Q$ are as follows:
$$
	\begin{array}{l l l}
	\mathcal F(Y_Q)	& =	 &\{  F_1^{0}, \ldots, F_{m+1}^{0}\} \cup \{Q_1, \ldots, Q_{k+1}\},\\
	\mathcal V(Y_Q)  	&= 	&\bigcup_{i=1}^{k+1} \mathcal V(Q_i),
	\end{array}
$$
where  $Q_i \cap Q_j \cap Y_Q$ is empty when $i \neq j$. 
Then  $Y_Q[Q_1, \ldots, Q_{k+1}]$ is a face-simple $(n+1)$-dimensional orientable manifold with marked facets $Q_1, \ldots, Q_{k+1}$.

We now define an  rs-characteristic function $\lambda$ on $Y_Q[Q_1, \ldots, Q_{k+1}]$.
Define  
$$
\lambda(F^0_j)=\begin{cases}\lambda'(F_j), & \textrm{for }j=1, \ldots, m\\
					      \lambda_0, &\textrm{for }j=m+1,
		         \end{cases} 
$$		         
where $\lambda_0$  is the vector obtained from Lemma \ref{indl}
in the following way.
For each vertex $V\in \mathcal V(Q_i)\subset \mathcal V(Y_Q)$ there exists a unique 
collection of facets
$F^0_{i_1}, \ldots, F^0_{i_n}\in \mathcal F(Y_Q)$ such that $V=F^0_{i_1}\cap\cdots\cap F^0_{i_n}\cap Q_i$.
Let $I_V:=\{i_1, \ldots, i_n\}\in \mathcal P_n[m]$ and let $\mathcal A:=\{I_V\mid V\in \mathcal V(Y_Q)\}$.
Then from Lemma~\ref{indl} there exist a primitive vector $\lambda_0\in \ZZ^n$ such that the set of $n$ vectors
$$\mathcal S_{i,V}=\{\lambda_0, \lambda(F^0_{i_1}), \ldots, \lambda(F^0_{i_{j-1}}), \widehat{\lambda(F^0_{i_j})}, \lambda(F^0_{i_{j+1}}), \ldots, \lambda(F^0_{i_n})\}$$
is linearly independent for all $j = 1, \ldots, n$ and $V \in \mathcal V(Y_Q)$.
One can check that $\lambda$ is indeed an rs-characteristic function on $Y_Q[Q_1, \ldots, Q_{k+1}]$.

On the other hand, $\tau'$ induces an orientable principal $T^n$-bundle $$\tau' \times \Id \colon E_X \times [0, 1] \to Q \times [0, 1].$$ 
So we have the pullback bundle $\tau \colon E_{Y_Q} = (\tau' \times \Id)^{-1}(Y_Q) \to Y_Q$ via the inclusion $Y_Q \hookrightarrow Y$. 
Note that $\tau \colon E_{Y_Q} \to Y_Q$ is an orientable smooth principal $T^n$-bundle. 
Then so is the pullback bundle  $\tau^i \colon E_i \to Q_i$  of $\tau$ via the inclusion $Q_i \hookrightarrow Y_Q$ for $i=1, \ldots, k+1$.
 So by Theorem~\ref{orbbd}, $W(Y_Q, \lambda, \tau)$ is an orientable effective $T^n$-orbifold 
whose boundary is the disjoint union $\bigcup_{i=1}^{k+1} (E_j/\sim_b)$ where $\sim_b$ is as defined in (\ref{equilam}). The space $E_{k+1}/\sim_b$
is  equivariantly diffeomorphic to $X(Q, \lambda', \tau')$, where $\tau' =\tau^{k+1} \colon E_X = E_{k+1} \to Q_{k+1}=Q$.

Let $\mathcal{F}(Q_i) =\{H_{i_1}, \ldots, H_{i_{n+1}}\} $. Then $ H_{i_j} = F_{i_j}^{0} \cap Q_i$ for a unique remaining  facet $F_{i_j}^{0}$ of $Y_Q [Q_1, \ldots, Q_{k+1}]$
for each $j=1,\ldots, n+1$.
Define a function $\lambda^i \colon \mathcal{F}(Q_i) \to \ZZ^n$ by 
$$\lambda^i(H_{i_j}) = \lambda(F^0_{i_j}) \quad \mbox{for} \quad j = 1, \ldots, n+1. $$ 
Then $\lambda^i$ is an r-characteristic function on $Q_i$ for all $i=1, \ldots, k$. 
So $\{(Q_i, \lambda^i)(E_i, Q_i, \tau^i)\}$ is a characteristic and topological data for $i=1, \ldots, k$. 
Note that the boundary component  $E_i/\sim_b$ is same as  $X(Q_i, \lambda^i, \tau^i) $, 
which is a quasitoric orbifold  over an $n$-simplex $Q_i$ for $i=1, \ldots, k$. 
Therefore by Lemma~\ref{cops}, each 
$E_i/\sim_b= X(Q_i, \lambda^i, \tau^i)$ is an orbifold complex projective space, which proves the theorem.
\end{proof}

We remark that in Theorem \ref{coborthm} if $X$ is a locally standard torus manifold then it is equivariantly cobordant to a disjoint union of some orbifold projective spaces. Still this theorem may be useful as there are examples of locally standard torus manifolds where it is quite complicated to compute cobordism invariabts, whereas there are several studies on topological invariants of orbifold projective spaces.

\begin{example}
Let $\bigtriangleup^2 = V_0V_1V_2$ be a triangle in $\RR^2$,  and let $C$ be a circle in the interior  of $\bigtriangleup^2$. 
Delete the open disk bounded by $C$ from $\bigtriangleup^2$, and let   $Q$ be the remaining subspace of $\bigtriangleup^2$.
A rational characteristic function $\lambda'$ on $Q$ is given as on the left of Figure \ref{egc1}. 
Let $Y :=Q \times [0,1]$.
Then the rs-characteristic function $\lambda$ on the  face-simple manifold with marked facets $Y_Q[Q_1,\ldots, Q_4]$ is given on the right of Figure \ref{egc1}, where
$Q_1,$ $Q_2$ and $Q_3$ are the facets of $Y_Q$ corresponding to the vertices $V_1 \times \{0\}, V_2 \times \{0\}$ and $V_3 \times \{0\}$ respectively, and $Q_4=Q\times \{1\}$. 
Since $H^2(Y_Q)=0$, any principal $T^2$ bundle $\tau$ over $Y_Q$ is trivial. 
By Theorem~\ref{orbbd}, the space $W(Y_Q, \lambda)$ is an orientable effective $T^2$-orbifold with boundary. Let $\mathcal{F}(Q_i) =\{H_{i_1}, H_{i_2}, H_{i_3}\} $ for $i=1,2,3$. 
Then $ H_{i_j} = F_{i_j}^{0} \cap Y_{Q_i}$ for a unique remaining facet 
$F_{i_j}^{0}$ of $Y_Q [ Q_1, \ldots, Q_{4}]$ for $i=1, 2, 3$. Define a function $\lambda^i \colon \mathcal{F}(Q_i) \to \ZZ^2$ by 
$$\lambda^i(H_{i_j}) = \lambda(F^0_{i_j}) \quad \mbox{for} \quad j = 1, 2, 3. $$ 
Then $\lambda^i$ is an r-characteristic function on $Q_i$ for $i=1, 2, 3$.
The r-characteristic function on $Q_4$ is same as on $Q$. Therefore, the boundary
 of $W(Y_Q, \lambda)$ is $$\bigsqcup_{i=1}^3 X(Q_i, \lambda^i) \bigsqcup X(Q, \xi)$$
where  $X(Q_i, \lambda^i)$ is an  orbifold complex projective space of real dimension $4$ for $i=1, 2, 3$.
\begin{figure}[ht]
        \centerline{
           \scalebox{0.80}{
            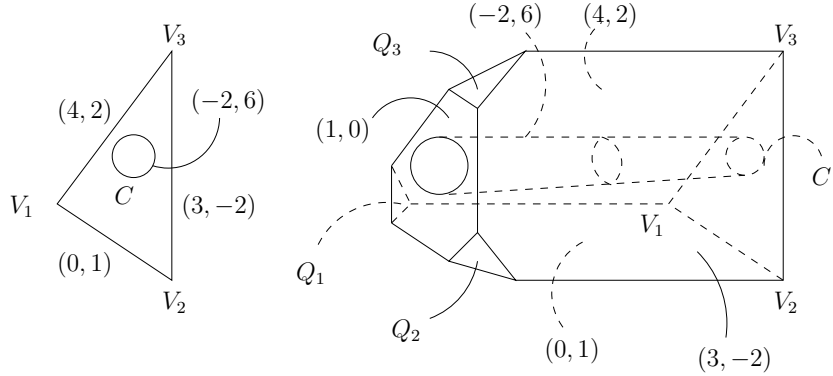
            }
          }
       \caption {An r- and rs-characteristic function of a 2- and 3-dimensional nice manifold with corners respectively.}
        \label{egc1}
      \end{figure}
      \qed
\end{example}
The following corollary is a special case of Theorem~\ref{coborthm} when $k=0$.
\begin{corollary}\label{coborthm2}
Let $X$ be an orientable  $2n$-dimensional locally standard torus orbifold without a fixed point. Then $X$ is the boundary of a $(2n+1)$-dimensional orientable effective orbifold with $T^n$-action.
\end{corollary}

\begin{corollary}\label{coborthm3}
Let $X$ be a $4$-dimensional orientable  locally standard torus manifold without a fixed point. Then $X$ is the boundary of an orientable $5$-dimensional manifold with  $T^2$-action.
\end{corollary}
\begin{proof} We are using the same notation as in the proof of Theorem~\ref{coborthm}.
Let $Q=X/T^2$ be the orbit space and $\{(Q, \lambda'), (E_X, \tau')\}$ be the combinatorial and topological data of $X$. 
Then  $X$ is $T^2$-equivariantly diffeomorphic to $X(Q, \lambda', \tau')$ by Theorem~\ref{clasi}. 
Since $X$ is $4$-dimensional, the orbit space $Q$ is a $2$-dimensional nice manifold with corners, and since $X$ has no fixed point, $Q$ has no vertices.
Therefore $Q$ is a bounded surface, and $Y_Q = Y = Q \times \bigtriangleup^1$.
Let $C_1, \ldots, C_k$ be the
boundary components of $Q$. By Corollary \ref{coborthm2}, $W(Y, \lambda, \tau)$ is an orientable effective $T^2$-orbifold with the boundary $X(Q,\lambda', \tau') \cong X$. 
Let $\pi \colon W(Y, \lambda,\tau) \to Y$ be the orbit map.
Notice that the orbifold singularity of $W(Y, \lambda, \tau)$ may occur only in $\bigsqcup_{i=1}^k \pi^{-1}(C_i \times \{0\})$. 
Let $U_i$ be a neighborhood of $C_i \times \{0\}$ in $Y$ for $i=1, \ldots, k$ such that:
\begin{enumerate}
	\item $U_i$ is diffeomorphic as manifold with corners to $C_i \times \RR^2_{\geq 0}$,

	\item $\bar{U_i} \subset Y$ is diffeomorphic as manifold with corners to $C_i \times \bigtriangleup^2$,

	\item $\bar{U_i} \cap \bar{U_j} = \emptyset$ and $\bar{U_i} \cap (Q \times \{1\}) = \emptyset$ for $i, j = 1, \ldots, k$ and $i \neq j$.
\end{enumerate}
Then $\bar{U_i} \cap (Y - \bigcup_{i=1}^k U_i)$ is diffeomorphic to $C_i \times \bigtriangleup^1$ as manifold with corners.
Let $W = W(Y, \lambda,\tau) - \bigcup_{1=1}^k \pi^{-1}(U_i)$, which is an orientable $T^2$-manifold with boundary. 
Then
$$\partial(W)= \bigsqcup_{i=1}^k ((T^2 \times C_i \times \bigtriangleup^1)/\sim_b)\bigsqcup X.$$
Note that $((T^2 \times C_i \times \bigtriangleup^1)/\sim_b)$ is $T^2$-equivariantly
 diffeomorphic to $C_i \times ((T^2 \times \bigtriangleup^1)/\sim_{b})$.
The facets of $C_i \times \bigtriangleup^1$ are $(C_i \times \bigtriangleup^1) \cap F_{i_j}$ for unique facets $F_{i_1}$ and $F_{j_2}$ of $Y$. The restriction of $\lambda$ on the facets of $\bigtriangleup^1 = c_i \times \bigtriangleup^1 \subset C_i \times \bigtriangleup^1$ is given by Figure \ref{torcob_2} $(a)$. 
\begin{figure}[ht]
        \centerline{
           \scalebox{0.80}{
            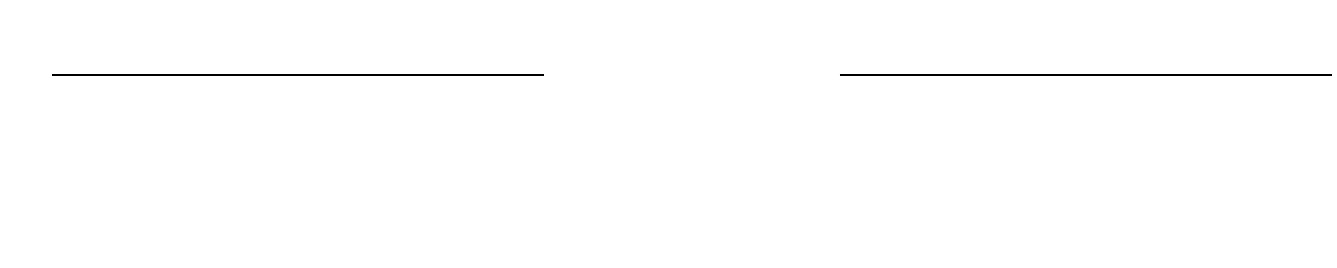
            }
          }
       \caption {}
        \label{torcob_2}
      \end{figure}
Note that vectors defined by $\lambda$ are all primitive. So we may assume $(a, b) = (1, 0)$ and $(c, d) = (-q_i, p_i)$ after an automorphism of $T^2$. Then we have the hyper characteristic function (see Definition 2.1 in \cite{SS}) on $\bigtriangleup^1$ given by Figure \ref{torcob_2} $(b)$. So following section 2 of \cite{SS}, one can get that $(T^2 \times \bigtriangleup^1)/\sim_{b}$ is weak-equivariantly diffeomorphic to the lens space $L(p_i, q_i)$ with the usual $T^2$ action for $i=1, \ldots, k$. 

%(XXXXX precise argument or at least some sketch of argument for the above statement is needed XXXXXX)

It is shown in \cite{Hanke} and \cite{SS} that any lens space $L(p_i, q_i)$ is $T^2$-equivariantly a boundary. 
Thus $C_i \times L(p_i, q_i)$ is a $T^2$-equivariantly boundary for each $i \in \{1, \ldots, k\}$, which proves the corollary. 
\end{proof}

We now give a couple of  equivariant cobordism results on orientable $4$-dimensional locally standard torus orbifolds.

\begin{theorem}\label{thm_4d-loc-tor}
 Let $M$ be a locally standard torus manifold over $2$-dimensional polytope $Q$ with $\partial{Q} \neq \emptyset$. 
 Then $M$ is equivariantly cobordant (by a $5$-dimensional $T^2$-manifold) to some copies of $\CC P^2$.
\end{theorem}

\begin{proof}
We stick to the notations of Lemma \ref{equi_classi}. Since $S^4$, $S^1 \times S^3$ and $T^2 \times S_Q$ are $T^2$-equivariantly boundaries, 
by Lemma~\ref{equi_classi}  $M$ is $T^2$-equivariantly cobordant to the connected sum $M_1 \sharp \cdots \sharp M_k$
of some $4$-dimensional quasitoric manifolds $M_1, \ldots, M_k$ by a $5$-dimensional $T^2$-manifold.  Since equivariant connected sum of two manifolds are equivariantly cobordant to the disjoint union of them,
$M_1 \sharp \cdots \sharp M_k$ is equivariantly cobordant to the disjoint union $~M_1 \sqcup\cdots \sqcup M_k$.
By Theorem 6.6 of \cite{Sar2}, we see that each $M_i$ is equivariantly cobordant to some copies of $\CC P^2$ by a $5$-dimensional $T^2$-manifold. 
This proves the theorem.
\end{proof}

In the following theorem we stick to the notations in Subsection~\ref{orb_hir_sur}. 

\begin{theorem}\label{bdd_orb_hir_sur}
Let $X$ be an orbifold  Hirzebruch surface whose  fan is given in Figure~\ref{egc5} $(B)$. 
If either $(a_1, b_1)=\pm (a_3, b_3)$ or $(a_2, b_2) = \pm (a_4, b_4)$, then $X$ is equivariantly a  boundary, i.e., $X$ is the boundary
of a $5$-dimensional  orientable $T^2$-orbifold.
\end{theorem}
\begin{proof}
Suppose $(a_1, b_1)=\pm (a_3, b_3)$. Let $P$ be the face-simple manifold  with marked facets $Q_1, Q_2$ and $Q_3$ as in Figure \ref{egc4} $(B)$.
Note that $Q_2$ and $Q_3$ are eye-shapes, and $P$ is orientable, smooth and contractible. Then we can define an rs-characteristic function $\eta$ on $P[ Q_1, Q_2, Q_3]$, as in Figure \ref{egc4} $(B)$. 
This induces an r-characteristic function $\eta^i$ on $Q_i$ for $i=1, 2, 3$.
\begin{figure}[ht]
        \centerline{
           \scalebox{0.80}{
            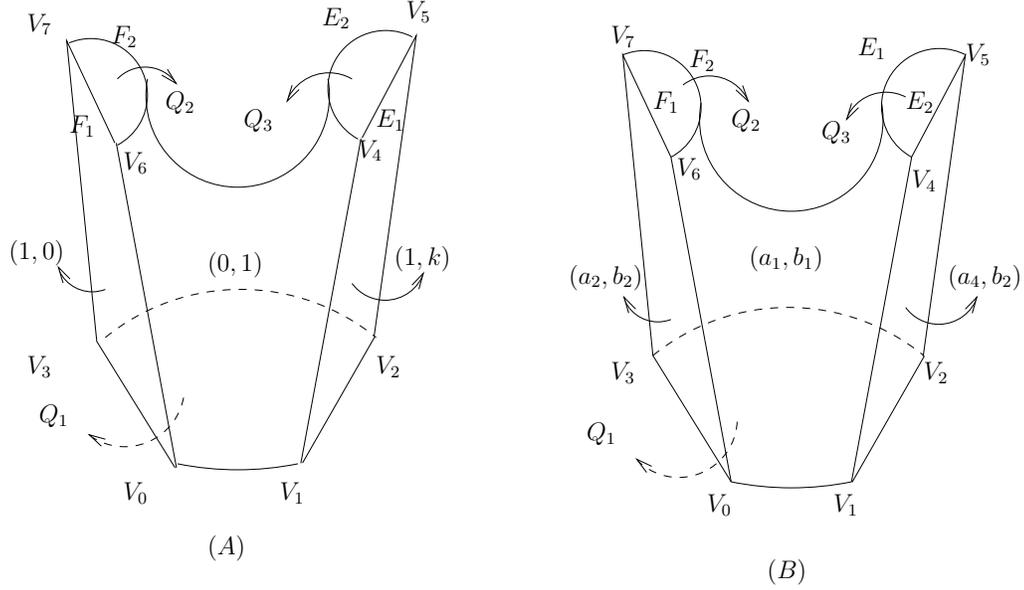
            }
          }
       \caption {An rs-characteristic functions on a face-simple $3$-dimensional manifold.}
        \label{egc4}
      \end{figure}
      
%Since $P \subset \RR^3$ is a $3$-dimensional manifold with corners and contractible in itself, $P$ is  an orientable nice manifold with corners. 
By Theorem~\ref{orbbd} the $T^2$-space $W(P, \eta)$ is a compact orientable
effective orbifold with  boundary  given by 
$$\partial(W(P, \eta))=X(Q_1, \eta^1) \sqcup X(Q_2, \eta^2) \sqcup X(Q_3, \eta^3)$$ 
where $X(Q_1, \eta^1)$ is $T^2$-equivariantly diffeomorphic to the orbifold Hirzebruch surface $X$ and $X(Q_i, \eta^i)$ is a torus orbifold over an eye-shape for $i=2, 3$. 
From Example \ref{bdd_lens} we can see that $X(Q_i, \eta^i)$ is $T^2$-equivariantly the boundary of a $5$-dimensional orientable orbifold with $T^2$ action for $i=2, 3$. 
This proves the theorem when $(a_1, b_1)=\pm (a_3, b_3)$.
The case when $(a_2, b_2) = \pm (a_4, b_4)$ is similar. 
\end{proof}

\begin{corollary}\label{cob_hir_sur}
If $X$ is a Hirzebruch surface, then $X$ bounds a $5$-dimensional orientable $T^2$-manifold.
\end{corollary}
\begin{proof}
Without loss of generality we may assume that  the complete rational fan of $X$ is as in Figure~\ref{egc5} $(A)$. 
Then the corresponding super characteristic function $\eta$ is given as in Figure~\ref{egc4} $(A)$. 
This induces  characteristic functions $\eta^i$ on $Q_i$ for $i=1, 2, 3$. 
%Since $P \subset \RR^3$ is a $3$-dimensional manifold with corners and contractible in itself, $P$ is orientable. 
Then by Corollary~\ref{manbdd}, we get an orientable $5$-dimensional $T^2$-manifold $W(P, \eta)$ with boundary given by  
$$\partial(W(P, \eta))=X(Q_1, \eta^1) \sqcup X(Q_2, \eta^2) \sqcup X(Q_3, \eta^3)$$ 
where $X(Q_1, \eta^1)$ is $T^2$-equivariantly diffeomorphic to the  Hirzebruch surface $X$ and $X(Q_i, \eta^i)$ is a locally standard torus manifold over an eye-shape for $i=2, 3$. So by Example~\ref{eye-shape2} $X(Q_i, \xi^i)$ is $T^2$-weakly equivariantly diffeomorphic to $S^4$ for $i=2, 3$. 
Therefore any Hirzebruch surface bounds a $5$-dimensional orientable $T^2$-manifold.
\end{proof}

\begin{remark}
Corollary \ref{cob_hir_sur} is proved in Lemma 6.1 in \cite{Sar2}, but we give a much shorter proof here. 
\end{remark}

We conclude this section by presenting  some explicit cobordism relations among orbifold complex projective spaces.
By definition, an orbifold complex projective space is a quasitoric orbifold over a simplex.
Let $Y$ be an $(n+1)$-dimensional orientable smooth manifold with corners in some euclidean space.
Let $\{F_1, \ldots, F_m\}$  be the facets and $\{V_1, \ldots, V_k\}$ be the vertices of $Y$.
Let $\VC(Y)$ be the vertex cut of $Y$, $F^{\prime}_j=F_j \cap Y_V$ and $Q_i$ be the facet of $VC(Y)$
corresponding to the vertex $V_i$ for $i=1, \ldots, k$. Then $Q_i$ is diffeomorphic to an $n$-simplex.
Applying Lemma \ref{indl}, we can define an rs-characteristic function $$\eta \colon \{F_1^{\prime}, \ldots,
F_m^{\prime}\} \to \ZZ^n$$ on $\VC(Y)[Q_1, \ldots, Q_k]$. Let $\mu \colon E \to VC(Y)$ be orientable principal $T^n$-bundle. Note that $\mu^{-1}(Q_i)$ is equivariantly diffeomorphic to $T^n \times Q_i$ for $i=1, \ldots, k$. Then by Lemma \ref{orbbd},
$W(\VC(Y), \eta, \mu)$ is an orientable effective $T^n$-orbifold with boundary where the boundary is the disjoint
union $\bigsqcup_{i=1}^k (T^n \times Q_i)/\sim_b$. Since restriction of the rs-characteristic function $\lambda$ to the 
facets of $Q_i$ is an r-characteristic function, $(T^n \times Q_i)/\sim_b$ is an orbifold complex projective space $O\mathbb{P}_i$ for $i=1, \ldots, k$. So, 
\begin{equation}\label{rel_in_oc}
[O\mathbb{P}_1] + \cdots + [O\mathbb{P}_k] = 0
\end{equation}
in the group $OC_{2n}$. At the end, it is natural to ask the following. 
\begin{ques}\cite[Question 5.9]{Sar3}
What are the other torus cobordism relations among the orbifold complex projective spaces?
\end{ques}

%{\bf Acknowledgement.} The authors would like to thank Mainak Poddar and Nigel Ray for many helpful discussion.
%The first author thanks Indian Institute of Sciences, Pacific Institute for Mathematical Sciences and
%University of Regina for support. He also thanks Indian Statistical Institute, 
%Kolkata and Institute of Mathematical Sciences for supporting his visiting fellowship.
%The second author is partially supported by the National Research Foundation of Korea(NRF) grant funded by the Korea government(MSIP) %(No. 2016R1A2B4010823).
\renewcommand{\refname}{References}

%\begin{thebibliography}{ALR}
%\bibliographystyle{alpha}
\bibliographystyle{abbrv}
\bibliography{bibliography.bib}
%\begin{thebibliography}{DJ}

\vspace{1cm}

\vfill

\end{document}